\documentclass[final]{article}   
\usepackage[a4paper, margin = 1.15in]{geometry}

\usepackage{amsmath, amssymb, amsthm}
\usepackage{mathtools, bbm, dsfont}
\usepackage{enumitem}
\usepackage{microtype}
\usepackage[indent]{parskip}
\usepackage{graphicx}
\usepackage[english,algoruled,lined,noresetcount,norelsize,linesnumbered]{algorithm2e}
\usepackage{setspace}

\usepackage[colorlinks, bookmarks, linkcolor=black, citecolor=black, urlcolor=black]{hyperref}
\usepackage[hyperpageref]{backref}
\renewcommand*\backref[1]{\ifx#1\relax \else 
(Return to page #1) \fi}
\usepackage[numbers,sort&compress]{natbib}
\usepackage[notcite,notref]{showkeys}
\def\mcite[#1]#2{\mbox{\cite[#1]{#2}}} 

\newtheorem{theorem}{Theorem}[section]
\newtheorem{lemma}[theorem]{Lemma}

\theoremstyle{definition}
\newtheorem{claim}[theorem]{Claim}
\newtheorem{definition}[theorem]{Definition}

\newtheorem{remark}[theorem]{Remark}

\setlist{leftmargin=*,labelindent=1em,itemsep=2pt,parsep=2pt}
\setenumerate{,itemsep=1pt,parsep=1pt}

\newcommand{\itmarab}[1]{\mbox{\rm 
({\it #1}\,\arabic{*}\hspace{0.05em})}}
\def\itm#1{\rm ({#1})} 
\def\itmit#1{\itm{\it #1\,}} 
\def\rom{\itmit{\roman{*}}}

\newcommand{\eps}{\varepsilon}

\newcommand{\im}{\mathrm{Im}}
\newcommand{\dom}{\mathrm{Dom}}
\newcommand{\dist}{\mathrm{dist}}
\newcommand{\exth}{\delta_{\rm e}(\Delta)}

\newcommand{\RG}{\normalfont{\textsc{rg}}}
\newcommand{\BL}{\normalfont{\textsc{bl}}}
\newcommand{\CNL}{\normalfont{\textsc{I}}}
\newcommand{\LG}{\normalfont{\textsc{g}}}
\renewcommand{\subset}{\subseteq}
\renewcommand{\epsilon}{\varepsilon}

\let\symmdiff\bigtriangleup
\let\hat\widehat
\let\tilde\widetilde

\def\cX{\mathcal{X}}

\newcommand{\oldqed}{}
\def\endofClaim{\hfill\scalebox{.6}{$\Box$}}
\newenvironment{claimproof}[1][Proof]{
  \renewcommand{\oldqed}{\qedsymbol}
  \renewcommand{\qedsymbol}{\endofClaim}
  \begin{proof}[#1]}
  {\end{proof}
  \renewcommand{\qedsymbol}{\oldqed}} 

\usepackage[mathlines]{lineno}
\usepackage{etoolbox} 

\newcommand*\linenomathpatch[1]{%
   \expandafter\pretocmd\csname #1\endcsname {\linenomath}{}{}%
   \expandafter\pretocmd\csname #1*\endcsname{\linenomath}{}{}%
   \expandafter\apptocmd\csname end#1\endcsname {\endlinenomath}{}{}%
   \expandafter\apptocmd\csname end#1*\endcsname{\endlinenomath}{}{}%
 }
\newcommand*\linenomathpatchAMS[1]{%
    \expandafter\pretocmd\csname #1\endcsname {\linenomathAMS}{}{}%
    \expandafter\pretocmd\csname #1*\endcsname{\linenomathAMS}{}{}%
    \expandafter\apptocmd\csname end#1\endcsname {\endlinenomath}{}{}%
    \expandafter\apptocmd\csname end#1*\endcsname{\endlinenomath}{}{}%
}

\expandafter\ifx\linenomath\linenomathWithnumbers
\let\linenomathAMS\linenomathWithnumbers
\patchcmd\linenomathAMS{\advance\postdisplaypenalty\linenopenalty}{}{}{}
\else
\let\linenomathAMS\linenomathNonumbers
\fi

\linenomathpatch{equation} 
\linenomathpatchAMS{gather}
\linenomathpatchAMS{multline}
\linenomathpatchAMS{align}
\linenomathpatchAMS{alignat}
\linenomathpatchAMS{flalign}

\title{Local Resilience for Containment of \\Bounded Degree Spanning
  Subgraphs\footnote{An extended abstract presenting the main result
    of this paper will appear in the post-proceedings of LATIN 2026,
    published as an LNCS volume (Springer).}}

\author{%
Peter Allen\thanks{Department of Mathematics, The London School of
Economics, Houghton Street, London WC2A 2AE, UK. E-mail: {\tt  \href{mailto:p.d.allen@lse.ac.uk}{p.d.allen@lse.ac.uk}, \href{mailto:j.boettcher@lse.ac.uk.}{j.boettcher@lse.ac.uk}, \href{mailto:m.s.neve@lse.ac.uk}{m.s.neve@lse.ac.uk}.}}
\and
Julia B\"{o}ttcher\footnotemark[2]
\and
Yoshiharu Kohayakawa\thanks{Instituto de Matem\'atica e
  Estat\'{\i}stica, Universidade de S\~ao Paulo, Rua do Mat\~ao 1010,
  05508–090 S\~ao Paulo, Brazil.  E-mail: {\tt yoshi@ime.usp.br}.
  Partially supported by FAPESP (2023/03167-5), CNPq (407970/2023-1,
  420838/2025-2, 315258/2023-3) and CAPES (Finance Code 001)}
\and
Mihir Neve\footnotemark[2]}

\date{}

\begin{document}
\setstretch{1.15}

\maketitle

\begin{abstract}
\vspace{5pt}
  We prove that for all~$\Delta \geq 2$ and~$\gamma > 0$, there exists a constant~$C = C(\Delta, \gamma)$ such that for $p\geq C(\log n/n)^{1/\Delta}$, asymptotically
  almost surely, every spanning subgraph~$G$
  of~$G(n,p)$ with minimum degree at least
  $\big(1-1/(2\Delta)+\gamma\big)pn$ contains every $n$-vertex graph~$H$ 
  with maximum degree at most~$\Delta$ and with at least~$Cp^{-2}$ vertices not in any triangles of $H$.  This is a `sparse local
  resilience version' of a classical theorem of Sauer and Spencer.
  
  The condition that~$H$ should contain some vertices not in triangles
  is necessary, and in fact the quantity~$p^{-2}$ is asymptotically best possible.  A
  key feature of our result is that~$H$ is allowed to be an expander
  graph, distinguishing it from previous results of similar nature,
  which dealt with, e.g., graphs of sublinear bandwidth.  Our proof
  makes use of regularity arguments, with the sparse blow-up lemma for
  random graphs being a key tool.
\end{abstract}

\section{Introduction}

The transference of results from extremal graph theory to sparse
settings has become an important area of research in the past two
decades (a comprehensive overview of these developments can for example be found in the survey~\cite{conlon14:_ICM}).  Taking a very broad view, this is a
particular chapter in an area of research that has contributed to the
proof of celebrated theorems, such as the Green--Tao theorem on the
existence of arbitrarily long arithmetic progressions in the
primes~\cite{green04:_prime_APs}.  

While the Green--Tao transference principle 
is concerned with transference to sparse pseudorandom settings, the
study of transference results to sparse, genuinely \textit{random}
settings is natural, and goes back to the 80s and 90s in the case of
graphs~\cite{frankl86} and arithmetic
progressions~\cite{kohayakawa96:_arith}.  In this paper, we transfer a
classic result of Sauer and Spencer~\cite{Sauer_spencer} to the sparse setting of
random graphs.

Within extremal graph
theory, transference results have been typically considered for Tur\'an- or Dirac-type theorems.  One such celebrated result, for example, is
the transference of Tur\'an's theorem to the binomial random graph
$G(n,p)$ by Schacht~\cite{Schacht}, and independently, by Conlon and
Gowers~\cite{ConGow}. Here, $G(n,p)$ is the random graph on~$n$
vertices obtained by including each of the possible $\binom{n}{2}$
edges independently with probability~$p$.  We say that an event holds
for $G(n,p)$ \emph{asymptotically almost surely} (a.a.s.)\ if the
probability of this event tends to~$1$ as~$n$ tends to infinity. The
authors of both the aforementioned papers determined the threshold for~$p$ above which the random graph $\Gamma\sim G(n,p)$ a.a.s. satisfies the following property for any
fixed constant $\epsilon>0$. If~$G$ is a subgraph
of~$\Gamma$ with at least $\bigl(1-1/(r-1)+\varepsilon\bigr)\bigl|E(\Gamma)\bigr|$
edges, then $G$~contains a copy of the complete graph~$K_r$. This result generalises Tur\'an's
theorem (the $p=1$ case) by transferring it to the setting of sparse random graphs.

In this paper, we focus on a Dirac-type transference result. Such results are often
phrased using the notion of `local resilience', which was introduced
by Sudakov and Vu~\cite{SudVu}.  Let~$\Gamma$ be an $n$-vertex graph that
 satisfies some property~$\Pi$. The \emph{relative local
  resilience}, or \textit{local resilience} for short, of~$\Gamma$
for~$\Pi$ is defined to be the largest $\alpha \in (0,1)$ such that every graph~$G$ obtained from~$\Gamma$ by deleting at most an $\alpha$
proportion of the edges incident to each vertex of~$\Gamma$ still
satisfies~$\Pi$.  For instance, Dirac's theorem implies that the
complete graph~$K_n$ ($n\geq3$) has local resilience~$\sim1/2$ for the
property of being Hamiltonian.  

In the sparse setting, local
resilience results for~$G(n,p)$ have been established for the property
of containment of perfect
matchings~\cite{ConEspKueOstKim,NenSteTru,SudVu}, Hamilton
cycles~\cite{ConEspKueOstKim,LeeSud,Montgomery,NenSteTru}, triangle
factors~\cite{BalLeeSam}, cycle factors~\cite{Trujic}, almost spanning
trees~\cite{BalCsaSam}, powers of Hamilton
cycles~\cite{FischSkoSteTru,SkoSteTru}, and bounded degree graphs of
sublinear bandwidth~\cite{AllBoeEhrSchTar,sparse_bandwidth}. A
common link in these local resilience results is that all  these graph
classes (perfect matchings, Hamilton cycles, etc.) only contain graphs
with sublinear bandwidth, thereby excluding, for instance, any graph
having good expansion properties. 

Here, we consider the class of graphs with maximum degree bounded
above by some fixed constant~$\Delta$.  In a classic result, Sauer and
Spencer~\cite{Sauer_spencer} determined a minimum degree condition on
the host graph~$G$ that enforces in~$G$ the containment of any
graph~$H$ having maximum degree at most~$\Delta$.  More precisely,
they prove the following result.

\begin{theorem}[Sauer--Spencer Theorem]
  \label{thm:Sauer_Spencer}
  Let $\Delta \in \mathbb{N}$ be given. Suppose~$G$ is an $n$-vertex graph with minimum degree
  $\delta(G) \ge(1-1/2\Delta)n$ and~$H$ is an $n$-vertex graph
  with maximum degree~$\Delta(H) \leq \Delta$. Then~$H$ is a spanning subgraph
  of~$G$.
\end{theorem}

Note that for any $\Delta\ge 3$, expander graphs on $n$ vertices and
with maximum degree $\Delta$ are known to exist for every $n$.  Hence, it is not
surprising that this class of graphs with bounded maximum degree is harder to handle than the graph classes mentioned above. In fact,
it is not yet known whether the minimum degree condition in Theorem~\ref{thm:Sauer_Spencer} is the best possible in enforcing the containment of
all graphs with maximum degree at most~$\Delta$. That is, even the $p=1$
case of the local resilience problem for the containment of bounded
degree graphs stands unsolved.  

A well-known conjecture of Bollob\'as, Eldridge, and
Catlin~\cite{Bollobas_Eldridge_version_conj,%
  catlin_version_conj} postulates that the minimum degree condition 
$\delta(G) \geq (1-1/(\Delta+1))n$ on $G$ is sufficient to enforce the containment of any graph~$H$ having $\Delta(H) \leq \Delta$. This conjecture is known to be tight, as is shown by a slightly unbalanced complete $(\Delta+1)$-partite
graph, which does not contain a $K_{\Delta+1}$-factor. While there have been several partial results towards the resolution of this conjecture (see, for example,~\cite{BEC_Delta2,BEC_Delta22,BEC_Delta3,BEC_bipartite,Kaul_kostochka_yu}),
the best known minimum degree condition over all values of~$\Delta$ is still given
by Theorem~\ref{thm:Sauer_Spencer}.

\subsection{Our Results}

We prove a local resilience result transferring
Theorem~\ref{thm:Sauer_Spencer} to the sparse setting of random graphs by showing
that the local resilience for containing a spanning
subgraph $H$ with $\Delta(H) \leq \Delta$ and enough vertices not in triangles in $G(n,p)$ is at
least~$1/2\Delta$. Note that in the random graph~$G(n,p)$, every vertex has~$(n-1)p$ neighbours in expectation, and in our setting $p$ is large enough that a.a.s.\ every vertex of $G(n,p)$ has degree very close to $pn$.

\begin{theorem}
    \label{thm:resil_main_basic}
    For all~$\Delta \geq 2$ and~$\gamma>0$, there is a
    constant~$C > 0$ such that for~$p \geq C(\log n/n)^{1/\Delta}$,
    the following holds asymptotically almost surely for~$\Gamma \sim G(n,p)$.
    Let~$G$ be a spanning subgraph of~$\Gamma$ with
    $\delta(G) \ge \big(1-1/2\Delta + \gamma\big)pn$, and let~$H$ be
    an $n$-vertex graph with $\Delta(H) \leq \Delta$ and with at
    least~$Cp^{-2}$ vertices not contained in any triangle of~$H$.  Then~$G$
    contains a copy of~$H$.
\end{theorem}

We remark that, technically, this theorem concerns a sequence
$H=(H_n)_{\,n\in\mathbb{N}}$ of graphs, but to simplify notation we just
write~$H$, as is standard in the area. Moreover, while we do not believe it to be optimal, the hypothesis on~$p$ that we
require above is the one under which the sparse blow-up lemma for random graphs~\mcite[Lemma 1.21]{sparse_blowup} operates, and is the natural benchmark up to which results embedding spanning graphs of maximum degree~$\Delta$ in~$G(n,p)$ are generally proved (see, for example, the universality results of~\cite{DKRR, Univ_Delta2}). 

Further, the requirement on the graph~$H$ to have at least~$Cp^{-2}$
vertices not contained in any triangles of~$H$ is indispensable, as
was shown by Huang, Lee, and Sudakov
in~\cite{HLS_vtsintriangles}. Indeed, they observe that already for sufficiently small 
constant values of~$p$, it is possible, after some cleaning-up of~$G(n,p)$, to delete all edges in the
neighbourhood of~$\Theta(p^{-2})$ vertices and obtain a
subgraph~$G$ satisfying the minimum degree condition mentioned in
Theorem~\ref{thm:resil_main_basic}.  Since such a subgraph~$G$ clearly
contains no~$H$ in which only~$o(p^{-2})$ vertices fail to belong to
triangles, the additional requirement on~$H$ in
Theorem~\ref{thm:resil_main_basic} is necessary.

As in the $p=1$ case, it is not known if
Theorem~\ref{thm:resil_main_basic} is optimal with respect to the
minimum degree condition $\delta(G) \geq (1 - 1/2\Delta +
\gamma)pn$. In fact, we are able to prove a more general version of
Theorem~\ref{thm:resil_main_basic}.  For this, we need to define the
following notion of \emph{extension threshold}.  In the definition
below, $H[S]$~denotes the subgraph of~$H$ induced by the
set~$S\subseteq V(H)$.

\begin{definition}[Extension threshold]
  \label{defn:extension_threshold}
  For all $\Delta \in \mathbb{N}$, the \emph{extension threshold}
  $\exth$ is the infimum over all real numbers~$\delta_{\rm e}$ for which the
  following holds.
  For every $\gamma>0$ there exists $\eta>0$ such that for all
  sufficiently large~$n$, if~$G$ and~$H$ are $n$-vertex
  graphs with $\delta(G)\ge(\delta_{\rm e}+\gamma)n$ and 
  $\Delta(H)\le\Delta$, and if $S\subset V(H)$ with $|S|\le\eta n$ is
  given with an embedding $\varphi_S: S\to V(G)$ of~$H[S]$ into~$G$,
  then there is an embedding $\varphi: V(H)\to V(G)$ of~$H$ into~$G$
  that extends~$\varphi_S$.
\end{definition}

It is easy to see that the infimum is attained in the above definition, and hence~$\exth$ also satisfies the extension property outlined in Definition~\ref{defn:extension_threshold}. The example witnessing the tightness of the Bollob\'as--Eldridge--Catlin
Conjecture clearly implies that $\exth\ge 1-1/(\Delta+1)$. Moreover, in a companion result on the robustness of Theorem~\ref{thm:Sauer_Spencer}~(see \cite{AllBoeKohNev:robust}), we prove that~$\exth\le 1-1/2\Delta$. In fact, we conjecture that the extension threshold~$\exth$ equals the lower bound $\Delta/(\Delta+1)$. Our main technical theorem is a (local) resilience result in
which the minimum degree condition on the graph~$G$ is given
by~$\exth$.  Indeed, we prove the following, which immediately
implies Theorem~\ref{thm:resil_main_basic}. 

\begin{theorem}
    \label{thm:resil_main}
    For all $\Delta \geq 2$ and $\gamma > 0$, there is a
    constant~$C > 0$ such that for $p \geq C(\log n/n)^{1/\Delta}$,
    the following holds asymptotically almost surely for $\Gamma \sim G(n,p)$.
    Let~$G$ be a spanning subgraph of~$\Gamma$ with
    $\delta(G) \geq (\exth + \gamma)pn$, and let~$H$ be an
    $n$-vertex graph with $\Delta(H) \leq \Delta$ and with at least
    $Cp^{-2}$ vertices not contained in any triangle of~$H$.  Then~$G$ contains
    a copy of~$H$.
\end{theorem}

Note that Theorem~\ref{thm:resil_main} implies that the local resilience of $G(n,p)$ for the containment of bounded degree spanning subgraphs with at least~$Cp^{-2}$ vertices not in any triangles is at least~$1 - \exth$. We further remark that for $\Delta = 1$, the lower and upper bounds on the extension threshold match, giving~$\delta_{\rm e}(1) = 1/2$. Hence, an analogous version of Theorem~\ref{thm:resil_main} for $\Delta = 1$ follows directly from the local resilience result for the containment of Hamilton cycles of Lee and Sudakov~\cite{LeeSud}.

The proof of Theorem~\ref{thm:resil_main} uses the sparse blow-up lemma for random
graphs from~\cite{sparse_blowup}, tools from~\cite{sparse_bandwidth}
for handling certain exceptional vertices, as well as a decomposition
result for bounded degree graphs from~\cite{AllBoeKohNev:robust}. In
contrast to~\cite{sparse_bandwidth}, the decomposition result is needed to handle graphs
with no hypothesis on their bandwidth.

\noindent \textbf{Organisation.} The remainder of this paper is organised as follows.  In Section~\ref{sec:prelim} we
introduce notation and tools needed in our proof.
Section~\ref{sec:Regularity_SBL} is dedicated to regularity
preliminaries, including the sparse blow-up lemma for random graphs.
Then, in Section~\ref{sec:sketch}, we give an overview of the proof of
Theorem~\ref{thm:resil_main} and collect some useful lemmas along the
way. This is followed by the proof of Theorem~\ref{thm:resil_main} in Section~\ref{sec:resil_proof}. 

\section{Notation and Main Tools}
\label{sec:prelim}

All graphs shall be assumed to be simple and finite. Most of the
graph-theoretic notation that we use is standard. We shall use $|G|$
to denote the size of the vertex set of $G$. Given two disjoint subsets
of vertices $A, B \subseteq V(G)$, we let~$E_G(A,B)$ 
be the set of edges of~$G$ with one endpoint in~$A$ and the other
endpoint in~$B$.  The minimum and maximum degrees of a graph~$G$ are denoted
by~$\delta(G)$ and~$\Delta(G)$, respectively. For any graph~$H$, an
$H$-factor in $G$ is a collection of vertex-disjoint copies of $H$ in
$G$, such that each vertex of $V(G)$ lies in exactly one of these
copies of $H$.

In our proofs, the letter~$G$ will be reserved for a subgraph of the random
graph~$\Gamma \sim G(n,p)$ of high minimum degree, into which a
bounded degree graph~$H$ needs to be embedded. The
variable~$n$ will typically denote the size of the vertex sets of $H$,
$G$, and $\Gamma$. For subsets $S,\, T \subseteq V(G)$, we
use~$N_G(S; T)$ to denote the neighbours, not necessarily common, of
the vertices of~$S$ in the graph~$G$ that lie in the set
$T\setminus S$.  For simplicity, we use $N_G(v; T)$ when~$S$ is a
singleton set~$\{v\}$, and~$N_G(S)$ when $T = V(G)$.  Further, we
shall use~$N^*_G(S; T)$, and analogously~$N^*_G(S)$, to denote the common
neighbours of the vertices in~$S$.  When $S=\emptyset$, this notation
should be understood to mean~$T$ and~$V(G)$,  respectively.  The
\emph{closed neighbourhood} $N_G(S) \cup S$
of~$S$ shall be denoted by~$N_G[S]$.  The
subscript~$G$ will be dropped in the neighbourhood notation
whenever the ambient graph is clear from the context.

Suppose~$H$ and~$G$ are two graphs given with a vertex map
$\varphi: V(H) \to V(G)$.  Then, for any edge $e = uv$ of~$H$, the
notation~$\varphi(e)$ shall be used to denote the potential
edge~$\varphi(u)\varphi(v)$ on the vertex set~$V(G)$. We shall say
that an injective map $\varphi: V(H) \rightarrow V(G)$ is an
\emph{embedding} of~$H$ into~$G$, if~$\varphi$ embeds a copy of~$H$
into~$G$ (i.e., $\varphi(E(H)) \subseteq E(G)$).

Given an $n$-element set~$X$ and a positive integer~$k$, a partition
$X_1 \sqcup \dots \sqcup X_k$ of~$X$ is said to be \emph{equitable} if
$|X_i| = \lfloor n/k \rfloor$ or $\lceil n/k\rceil$ for all
$i \in [k]$, that is, all the parts~$X_i$ of the partition are of size
as equal as possible. Finally, we use the notation $x = a \pm \delta$ [similarly, $x \neq a \pm \delta$] for $x, a \in \mathbb{R}$ and~$\delta > 0$, to concisely denote the statement $x \in (a-\delta, a+\delta)$ [similarly $x \notin (a-\delta, a+\delta)$]. By the notation~$a/bc$, we mean $a/(bc)$.

We will require the following theorem of Hajnal and Szemer\'{e}di~\cite{Hajnal_szemeredi}, which allows for partitioning a graph with bounded maximum degree into independent sets of nearly equal sizes.  

\begin{theorem}[Hajnal--Szemer\'{e}di theorem]
\label{thm:Hajnal-Szemeredi}
Let $\Delta \in \mathbb{N}$ be given. If\/ $H$ is a graph with maximum degree~$\Delta(H) \leq \Delta$, then\/ $H$ has an equitable partition into\/ $\Delta + 1$ independent sets.     
\end{theorem}

Theorem~\ref{thm:Hajnal-Szemeredi} has many useful applications. For instance, given a graph~$H$ with $\Delta(H) \leq \Delta$ and some constant~$C$, we shall use Theorem~\ref{thm:Hajnal-Szemeredi} to obtain an equitable partition of $V(H)$ with parts of size~$\Theta(n)$ and such that for every pair of vertices~$x,y$ within a part, we have $\dist_H(x,y) \geq C$.

\subsection{Probabilistic Tools}

Next, we gather some probabilistic tools that will be useful in our proof of Theorem~\ref{thm:resil_main}. We begin with a Chernoff bound for hypergeometrically distributed random variables. Recall that a random variable~$X$ is said to follow a \emph{hypergeometric distribution} with parameters $(n, m, k)$ 
if given a set $\mathcal{U}$ of size $n$ and a fixed subset $\mathcal{A} \subseteq \mathcal{U}$ of size $m$, the random variable $X$ counts the size of the intersection~$\mathcal{A} \cap \mathcal{T}$ for a subset $\mathcal{T} \subseteq \mathcal{U}$ of size~$k$ generated uniformly at random among all the~$\binom{n}{k}$~such subsets of~$\mathcal{U}$. Note that the expected value $\mathbb{E}(X)$ of~$X$ equals~$mk/n$. We use the following concentration bound for~$X$, which, while weaker than the standard Chernoff bound, is a more convenient form for our application (see, e.g., \mcite[Theorem 2.1 and Theorem 2.10]{Random_graphs_textbook_janson}). 

\begin{lemma}
[Hypergeometric Chernoff bound]
\label{lem:chern_hypgeom}
Let $X$ be a hypergeometrically distributed random variable with parameters $(n, m, k)$. Then for any $\eps \in (0,1)$ and for $t \geq \eps\, \mathbb{E}(X)$, we have
\[
\mathbb{P}\Bigl(\bigl|X - \mathbb{E}(X)\bigr| \geq t \Bigr) \leq
2e^{-\eps^2t/3}\,.
\]
\end{lemma}
 
We shall also use the following property of random graphs. Roughly speaking, it states that with probability tending to $1$, nearly all vertices in $G(n,p)$ have approximately the expected number of neighbours within large enough subsets of $V(G)$. See \cite[Proposition 19]{sparse_bandwidth} for a short proof. 

\begin{lemma}
\label{lem:rand_G_property}
For every $\eps>0$ there exists a constant $C_{\RG} >0$ such that for every $0<p=p(n)<1$, the following holds  asymptotically almost surely for~$\Gamma\sim G(n,p)$. For every $X \subseteq V(\Gamma)$ with $|X| \geq C_{\RG}\, p^{-1} \log n$, there are at most $C_{\RG}\, p^{-1} \log \bigl(en/|X|\bigr)$ vertices $v \in V(\Gamma)$ that satisfy
\[
  \bigl||N_{\Gamma}(v;X)| - p |X|\bigr| > \eps p
  \bigl|X\bigr| .
\]
\end{lemma}

\section{Regularity and the Sparse Blow-up Lemma}
\label{sec:Regularity_SBL}

We now turn to the notions of regularity and the regularity method. 
In his celebrated regularity lemma, Szemer\'{e}di~\cite{Sz_regularity_OG} showed that the vertex set of any large graph can be equitably partitioned into a bounded number of parts, such that for most pairs of parts the edges crossing this pair are distributed ``random-like''.
The regularity lemma has been an extensively used tool in extremal
graph theory, often in conjunction with the blow-up lemma of
Koml\'{o}s, S\'{a}rk\"{o}zy, and Szemer\'{e}di~\cite{Blowup_lemma},
which allows the embedding of certain bounded degree spanning subgraphs into~$G$, provided that we can find a suitable structure in the regular partition given by the regularity lemma.

Our proof relies on sparse versions of the regularity lemma and the blow-up lemma. More precisely, we shall use a consequence of such a sparse regularity lemma established in~\cite{sparse_bandwidth}, which we state in Section~\ref{sec:sketch} (see Lemma~\ref{lem:resil_G}). Hence, we shall not explicitly formulate a sparse regularity lemma here. We do, however, need to formulate the sparse blow-up lemma, as proved in \cite{sparse_blowup}. This requires a series of definitions.

 For $0 < p < 1$, the \emph{$p$-density} of the pair $(A,B)$ is defined as $d_{G,p}(A,B) \coloneqq |E_G(A,B)|/(p|A||B|)$. For an aptly chosen $p$, this normalisation allows for the comparison of the density of a pair to the density of the entire graph; in particular when we think of~$G$ as a subgraph of a random graph~$\Gamma \sim G(n,p)$, as will be the case for us. 

 The sparse $\eps$-regularity notion we shall use is as follows.
 The pair $(A,B)$ is \emph{$(\eps,d,p)$-lower-regular} in~$G$ if for all $A'\subseteq A$ with $|A'|\ge\eps|A|$ and $B'\subseteq B$ with $|B'|\ge\eps |B|$, we have $d_{G,p}(A',B') \geq d - \eps$.
 The pair~$(A,B)$ in $G\subseteq\Gamma$ is \emph{$(\eps,d,p)$-super-lower-regular} in~$G$ if it
  is $(\eps,d,p)$-lower-regular and for every $u\in A$ and $v\in B$ we have
\begin{align*}
    \deg_G(u;B) &>(d-\eps) \max\{p |B|,\tfrac12\deg_{\Gamma}(u;B)\}\,\,\text{ and}\\
    \deg_G(v;A) &>(d-\eps) \max\{p |A|, \tfrac12\deg_{\Gamma}(v;A)\}\,.
 \end{align*}
 
 Note that in the definitions above we only require a lower bound
 on~$d(A', B')$. This differs from the usual definition of $\eps$-regularity, and is sometimes called \emph{dense} in the literature, for example~\cite{dense_reg}. For better readability, throughout this paper, we shall use the term \emph{regularity} to mean \emph{lower-regularity}, and shall use the latter whenever a distinction is useful for emphasis or clarity. 
The following lemma states that lower-regular pairs are
 ``robust'' in view of small alterations of the respective vertex
 sets. For a proof, see for example~\cite[Proposition~12]{sparse_bandwidth}.

\begin{lemma}
  \label{lem:sparse_reg_robust}
  Let $(X,Y)$ be an $(\eps,d,p)$-lower-regular pair in a graph $G$ and
  let $\hat{X}$ and $\hat Y$ be two subsets of $V(G)$ such that
  $|X \symmdiff \hat{X}| \leq \mu |X|$ and
  $|Y \symmdiff \hat Y| \leq \nu |Y|$ for some $0 \leq \mu, \nu \leq
  1$. Then $(\hat X, \hat Y)$ is $(\hat \eps, d, p)$-lower-regular for $\hat \eps \coloneqq \eps + 2\sqrt{\mu} + 2 \sqrt{\nu}$.
\end{lemma}

In contrast to the dense setting, super-regularity is not sufficient
to obtain a blow-up lemma in the sparse setting. For instance, given
three pairwise $(\eps, d, p)$-super-regular sets $A$, $B$, and $C$, it
is possible that for some $a \in A$, the sets $N_G(a;B)$ and
$N_G(a;C)$ have no edges between them, which is problematic for
embedding graphs. For instance, this would prevent us from embedding a
triangle factor in the tripartite graph~$G[A,B,C]$. To remedy this,
the sparse blow-up lemma additionally requires the following notions
of regularity inheritance, which ensure that this problem cannot
occur.

\begin{definition}[Regularity inheritance]
\label{defn:reg_inheritance}
  Let $A$, $B$, and $C$ be vertex sets in the subgraph~$G$
  of~$\Gamma \sim G(n,p)$, where~$A$ and~$B$ are disjoint and~$B$ and~$C$
  are disjoint, but we do allow $A=C$. We say that $(A,B,C)$ has
  \emph{one-sided $(\eps,d,p)$-inheritance} if for each $u\in A$ the pair
  $\big(N_\Gamma(u;B),C\big)$ is $(\eps,d,p)$-regular.
  If in addition~$A$ and~$C$ are
  disjoint, then we say that $(A,B,C)$ has \emph{two-sided
    $(\eps,d,p)$-inheritance} if for each $u\in A$ the pair
  $\big(N_\Gamma(u;B),N_\Gamma(u;C)\big)$ is $(\eps,d,p)$-regular.
\end{definition}

The sparse version of the blow-up lemma, as given
in~\cite{sparse_blowup}, and unlike the dense blow-up lemma
from~\cite{Blowup_lemma}, does not require all $(\eps, d, p)$-regular
pairs in a regular partition of~$G$ to be super-regular. Instead, it
works with the whole regular partition (given by the so-called reduced
graph~$R$) as well as a substructure of this partition (given by a
spanning subgraph~$R'$ of~$R$). For this to work, roughly speaking,
the sparse blow-up lemma further requires as input a
family~$\tilde\cX$ of small sets of vertices of~$H$ such that the
two-sided inheritance condition, as defined above, is only required
with respect to the vertices in the members of~$\tilde\cX$.  The vertices of $\tilde{\mathcal{X}}$ are called \textit{buffer vertices}.  To make the setup precise, we introduce several
concepts in Definition~\ref{defn:sparseBL_generic_terms} below.

However, before we proceed, let us remark that in our proof of Theorem~\ref{thm:resil_main}, we shall not apply the
sparse blow-up lemma to embed~$H$ in~$G$ directly. We shall instead have to
embed a small set of vertices $X_E\subset V(H)$ onto some small set
$V_E\subset V(G)$ first, and then we shall be in position to apply the
blow-up lemma to embed $H'=H-X_E$ in $G'=G-V_E$.  Therefore, in what
follows, our definitions are given for graphs~$H'$ and~$G'$.

\begin{definition}
\label{defn:sparseBL_generic_terms} 
Let $G'$ and $H'$ be two graphs, given with partitions
$\mathcal{V}=\{V_i\}_{i\in[r]}$ and $\mathcal{X}=\{X_i\}_{i\in[r]}$ of their respective vertex sets. 
Let $R$ and $R'\subset R$ be two graphs on $r$ vertices.
Let $\tilde{\mathcal{X}}=\{\tilde{X}_i\}_{i\in[r]}$ be a family of subsets $\tilde{X}_i\subset X_i$. Let $0<\eps,\alpha,d,p<1$ be given.
\begin{itemize}
\item[$\bullet$] $\mathcal{V}$ and~$\mathcal{X}$ are said to be \emph{size-compatible} if $|V_i|=|X_i|$ for each $i\in[r]$.
\item[$\bullet$] For $\kappa\geq 1$, we say that $(G',\mathcal{V})$ is
  \emph{$\kappa$-balanced} if there exists $m\in\mathbb{N}$ such that
  $m\leq |V_i|\leq \kappa m$ for all $i\in[r]$. This is defined
  analogously for $(H', \mathcal{X})$. 
\item[$\bullet$] $(G',\mathcal{V})$ is said to be \emph{$(\eps,d,p)$-regular on~$R$} if
  $(V_i,V_j)$ is $(\eps,d,p)$-regular for each $ij\in E(R)$.    In
  this case, we also say that~$R$ is a \emph{reduced graph}
  for~$\mathcal{V}$ and call the parts of $\mathcal{V}$
  \emph{clusters}.
  
\item[$\bullet$] $(G',\mathcal{V})$ is said to be
  \emph{$(\eps,d,p)$-super-regular on~$R'$} if $(V_i,V_j)$ is
  $(\eps,d,p)$-super-regular for each $ij\in E(R')$. 

\item[$\bullet$] $(H',\mathcal{X})$ is an \emph{$R$-partition} if whenever there are
  edges of~$H'$ between~$X_i$ and~$X_j$, the pair $ij$ is an edge
  of~$R$.  
\item[$\bullet$] $(G',\mathcal{V})$ is said to have \emph{one-sided
    inheritance on~$R'$} if $(V_i,V_j,V_k)$ has one-sided
  $(\eps,d,p)$-inheritance for every pair of edges $ij,\,jk\in E(R')$.
\item[$\bullet$] $\tilde{\mathcal{X}}=\{\tilde{X}_i\}_{i\in[r]}$ is
  said to be an \emph{$(\alpha,R')$-buffer} for $(H',\mathcal{X})$ if
  for each $i\in[r]$
  \begin{enumerate}[label=\rom]
   \item $|\tilde{X}_i|\ge\alpha |X_i|$, and
   \item for each $x\in\tilde{X}_i$, the first and second
     neighbourhoods of~$x$ \emph{go along~$R'$}, that is, for each
     $xy,\,yz\in E(H')$ with $y\in X_j$ and $z\in X_k$ we have
     $ij\in E(R')$ and $jk\in E(R')$.
  \end{enumerate}
\item[$\bullet$] $(G',\mathcal{V})$ is said to have \emph{two-sided
    inheritance on~$R'$ for $\tilde{\mathcal{X}}$} if $(V_i,V_j,V_k)$
  has two-sided $(\eps,d,p)$-inheritance whenever there is a triangle
  $x_ix_jx_k$ in~$H'$ with $x_i\in \tilde{X}_i$, $x_j\in X_j$, and
  $x_k\in X_k$.
\end{itemize}
\end{definition}

The sparse blow-up lemma further affords so-called image restrictions,
which allow us to restrict the image of some vertices in the desired
embedding (vertices in the sets~$X_i^*$ in the definition that
follows).  These restrictions usually result from a pre-processing step in which some
vertices of the graph being embedded are mapped to some badly behaved
vertices of the host graph. 

More concretely, recall that we wish to
embed an $n$-vertex graph~$H$ into an $n$-vertex graph~$G$, and
suppose that we have a small set of vertices $V_E\subset V(G)$ that we
need to use in the embedding first (the badly behaved vertices).  In
such a scenario, we shall specify a set of
vertices $X_E\subset V(H)$ to map onto~$V_E$ in our desired
embedding.  This is the pre-processing step we mentioned above.  We
shall then use the sparse blow-up lemma to embed $H'=H-X_E$ into
$G'=G-V_E$.  Note that this embedding of~$H'$ into~$G'$ has to take
into account the way we mapped the vertices of~$X_E$ onto~$V_E$.
This requirement is encoded via image restrictions, which play a crucial role in our application of the sparse blow-up lemma.

Note that any vertex $x \in V(H')$ that is to be embedded may be adjacent to vertices in $X_E$ that have been pre-embedded to a set of vertices~$J_x \subseteq V_E$. These vertices in $J_x$ can restrict the candidate set~$I_x \subseteq V(G')$ for embedding the vertex~$x$. The sets~$J_x \subset V(\Gamma)\setminus V(G')$ and $I_x \subseteq V(G')$ will be referred to as \emph{restricting vertices} and \emph{image restrictions}, respectively. The requirements for image restrictions use an ambient
graph~$\Gamma$ (which in our application will be a random
graph~$G(n,p)$) and a density constant~$d>0$.  The set of pre-embedded vertices in~$\Gamma$ is given by the set~$V_E = V(\Gamma)\setminus V(G')$, where~$G'$ is the host graph to be
considered later in the sparse blow-up lemma.    

\begin{definition}[Image restrictions]\label{defn:image_restrict}  
  Let $\eps$, $d$, $p$, and $(H',\mathcal{X})$ and $(G',\mathcal{V})$ be as in
  Definition~\ref{defn:sparseBL_generic_terms} and
  $\Gamma\supseteq G'$.  Let $\mathcal{I}=\{I_x\}_{x\in V(H')}$ and $\mathcal{J}=\{J_x\}_{x\in V(H')}$ be
  collections of subsets of~$V(G')$ and
  $V(\Gamma)\setminus V(G')$, respectively.  Let
  $X_i^* \subseteq X_i$ denote the set of those vertices $x \in X_i$
  for which ${I_x\neq V_i}$, which we call the \emph{image-restricted}
  vertices in~$X_i$.  We say that~$\mathcal{I}$ and~$\mathcal{J}$ form
  a \emph{$(\rho, \zeta, \Delta, \Delta_J)$-restriction pair} if for
  each $i,\,j\in[r]$ and $x\in X_i$, we have
  \begin{enumerate}[label=\itmarab{IR}]
    \item\label{itm:IR1} $|X_i^*|\leq\rho|X_i|,$ 
    \item\label{itm:IR2} if $x \in X_i^*$, then $I_x\subseteq
    N^*_\Gamma(J_x;V_i)$ and $|I_x|\ge\zeta(dp)^{|J_x|}|V_i|$,
    \item\label{itm:IR3} if $x\in X_i^*$, then $|J_x|+\deg_{H'}(x)\le\Delta$ and 
    if $x\not\in X_i^*$, then $J_x=\emptyset$, 
    \item \label{itm:IR4} each vertex of~$\Gamma$ appears in at
      most~$\Delta_J$ sets of~$\mathcal{J}$, 
    \item\label{itm:IR5} 
    $\big|N^*_\Gamma(J_x;V_i)\big|=(p\pm\eps p)^{|J_x|}|V_i|$, and
    \item\label{itm:IR6} the pair $\bigl(N^*_\Gamma(J_x;V_i),N^*_\Gamma(J_y;V_j)\bigr)$ is $(\eps,d,p)$-regular in~$G'$ for each $xy \in E(H')$ with $x\in X_i^*$ and
      $y\in X_j$.
 \end{enumerate}
\end{definition}

We now state the sparse blow-up lemma for random graphs~\cite[Lemma 1.21]{sparse_blowup}. Note that the graph~$G'$ in Lemma~\ref{lem:rg_blowup} below need not be a
spanning subgraph of~$\Gamma$.

\begin{lemma}[Blow-up lemma for $G(n,p)$~\cite{sparse_blowup}]
  \label{lem:rg_blowup}
  For all $\Delta \geq 2$,~$\Delta_{R'}$,~$\Delta_J$,
  $\alpha,\zeta, d>0$, and~$\kappa>1$, there exist~$\eps_{\BL},\rho>0$
  such that for all~$r_1$ there exists a constant~$C_{\BL}$, such that for~$p>C_{\BL}\big(n^{-1}\log n\big)^{1/\Delta}$, the random graph~$\Gamma \sim G(n,p)$ a.a.s. satisfies the following.
   
  Let~$R$ be a graph on~$r\le r_1$ vertices and let~$R'\subset R$ be a spanning
  subgraph of~$R$ with~$\Delta(R')\leq \Delta_{R'}$.
  Let~$H'$ with~$\Delta(H')\leq \Delta$ and $G'\subset \Gamma$ be graphs with $\kappa$-balanced
  size-compatible vertex partitions~$\mathcal{X}=\{X_i\}_{i\in[r]}$ and~$\mathcal{V}=\{V_i\}_{i\in[r]}$, respectively, which have
  parts of size at least~$n/(\kappa r_1)$.
  Let~$\tilde{\mathcal{X}}=\{\tilde{X}_i\}_{i\in[r]}$ be a family of subsets of~$V(H')$, 
  let~$\mathcal{I}=\{I_x\}_{x\in V(H')}$ be a family of image restrictions, and 
  let~$\mathcal{J}=\{J_x\}_{x\in  V(H')}$ be a family of restricting vertices.
  Suppose that
  \begin{enumerate}[label=\itmarab{BL}]
  \item\label{itm:SBL1} $(H',\mathcal{X})$ is an $R$-partition, and
    $\tilde{\mathcal{X}}$ is an $(\alpha,R')$-buffer for
    $(H',\mathcal{X})$;
  \item\label{itm:SBL2} $(G',\mathcal{V})$ is
    $(\eps_{\BL},d,p)$-regular on~$R$ and
    $(\eps_{\BL},d,p)$-super-regular on~$R'$, and has one-sided
    inheritance on~$R'$ and two-sided inheritance on~$R'$ for
    $\tilde{\mathcal{X}}$; and
  \item\label{itm:SBL3} $\mathcal{I}$ and $\mathcal{J}$ form a
    $(\rho,\zeta,\Delta,\Delta_J)$-restriction pair.
  \end{enumerate}
  Then there is an embedding $\varphi\colon V(H')\to V(G')$ such that
  for each $x\in V(H')$, we have $\varphi(x)\in I_x$.
\end{lemma}

\section{Proof Sketch and Key Lemmas}
\label{sec:sketch}

In this section, we sketch an outline of our proof of Theorem~\ref{thm:resil_main} and collect some useful
lemmas along the way.
Given $\Gamma \sim G(n,p)$, suppose that the graphs~$G\subset\Gamma$ and~$H$ satisfy the conditions of
Theorem~\ref{thm:resil_main}. Our ultimate aim is to be able to pre-process the graphs~$G$
and~$H$ in a manner that reduces Theorem~\ref{thm:resil_main} to an application of the sparse blow-up lemma (Lemma~\ref{lem:rg_blowup}). Broadly speaking, this will be achieved by first obtaining a sparse-regular partition for the graph~$G \subseteq \Gamma$. Given such a regular partition for $G$, we shall partition the vertices of~$H$ and identify buffer vertices in a manner that satisfies the requirements of Lemma~\ref{lem:rg_blowup}. At this stage, if all conditions of Lemma~\ref{lem:rg_blowup} are met, then Theorem~\ref{thm:resil_main} would follow directly by an application of Lemma~\ref{lem:rg_blowup}.

However, while obtaining a sparse-regular partition for the graph~$G$, it is possible that there exists a small set of exceptional vertices $V_0 \subseteq V(G)$ that are ``badly behaved'' with respect to this partition. In the sparse setting, these exceptional vertices of $V_0$ will never satisfy the conditions of Lemma~\ref{lem:rg_blowup}, and thus need to be dealt with separately. We handle these exceptional vertices by \emph{pre-embedding} some vertices of the graph~$H$ into distinct vertices of~$V_0$. This pre-embedding will impose certain restrictions on the possible images of other vertices of~$H$, and hence needs to be done in a manner such that the resulting image restrictions satisfy condition~\ref{itm:SBL3}. Now, an application of Lemma~\ref{lem:rg_blowup} along with the pre-embedding returns a copy of~$H$ in~$G$, as required.   

Let us now give a more detailed outline of the proof of Theorem~\ref{thm:resil_main}. The majority of the proof builds on ideas and lemmas developed
in~\cite{sparse_bandwidth}. However, as we are not restricted to
graphs of small bandwidth, we also need an additional decomposition result for general bounded degree graphs (Lemma~\ref{lem:Lemma_H} below). 

Our proof of Theorem~\ref{thm:resil_main} begins with obtaining a
sparse-regular partition for the graph~$G$ with the help of the following
lemma, which follows from \cite[Lemma 26]{sparse_bandwidth} (by setting
$r_0 = 1$ there, and omitting the so-called backbone graph
$B^K_r$). For the given graph~$G$, it returns a reduced graph~$R$ of
high minimum degree containing a clique-factor~$R'$, an exceptional
set of vertices~$V_0$ of size at most~$\mathcal{O}(p^{-2})$, and a
partition~$\mathcal{V} = \{ V_i\}_{i \in [|R|]}$ of the remaining
vertices of~$G$. This partition is balanced and satisfies
regularity inheritance conditions. A partition
$\{V_{ij}\}_{i \in [m], j\in [K]}$ is said to be \emph{$K$-equitable}
if $\bigl| |V_{ij}|-|V_{ij'}|\bigr| \leq 1$ for all $i \in [m]$ and
$j,\,j' \in [K]$. 

\begin{lemma}[Lemma for $G$~\cite{sparse_bandwidth}] 
  \label{lem:resil_G}
  For each $\gamma > 0$, integer $K \geq 2$, and $\xi \geq 1 - 1/K$, there exists $d_0 > 0$
  such that for every $\eps \in (0,1/2K)$, there
  exist $r_1\geq 1$ and $C_{\LG}>0$ such that the
  following holds a.a.s. for $\Gamma \sim G(n,p)$ with
  $p \geq C_{\LG} (\log n/n)^{1/2}$.

  Let $G=(V,E)$ be a spanning subgraph of $\Gamma$ with
  $\delta(G) \geq (\xi + \gamma)pn$. Then there exists
  an integer $m$ with $r = mK \leq r_1$, a subset $V_0 \subseteq V$
  with $|V_0| \leq C_{\LG}p^{-2}$, a $K$-equitable vertex partition $\mathcal{V} = \{V_{ij}\}_{i\in[m],j\in[K]}$ of $V(G)\setminus V_0$,
  and an $r$-vertex graph $R$ on the vertex set $[m] \times [K]$ with
  $\delta(R) \geq (\xi + \gamma/2)r$ and
  containing a $K$-clique factor $R'$ (whose cliques have vertex set~$\{i\} \times [K]$ for all $i \in [m]$), such that the following hold for all $d \leq d_0$.
  \begin{enumerate}[label=\itmarab{G}]
  \item \label{itm:RG1} For every $i\in[m]$ and $j\in[K]$, we have $n/4r\leq (1- 2\eps)n/r\leq
    |V_{ij}| \leq4n/r$,
  \item \label{itm:RG2} $(G,\mathcal{V})$ is
    $(\eps,d,p)$-lower-regular on $R$ and $(\eps,d,p)$-super-regular
    on $R'$,
    
  \item \label{itm:RG3} the pairs $\bigl(N_\Gamma(v; V_{ij}),V_{i'j'}\bigr)$ and
    $\bigl(N_\Gamma(v; V_{ij}),N_\Gamma(v; V_{i'j'})\bigr)$ are
    $(\eps,d,p)$-lower-regular pairs in~$G$ for every
    $\{(ij),(i'j')\} \in E(R)$ and $v\in V\setminus V_0$, and
    
  \item \label{itm:RG4} $|N_\Gamma(v;V_{ij})| = (1 \pm
    \eps)p|V_{ij}|$ for every $i \in [m]$, $j\in [K]$ and, $v
    \in V \setminus V_0$.
    
  \end{enumerate}
\end{lemma}

\begin{remark}
    We remark that Lemma~\ref{lem:resil_G} has been slightly adapted from \cite[Lemma 26]{sparse_bandwidth} in accordance with our requirements for the proof of Theorem~\ref{thm:resil_main}. Please refer to Appendix~\ref{app:adaptLemG} for further details.  
\end{remark}

 The partition of~$G$ returned by Lemma~\ref{lem:resil_G} contains a small set of exceptional vertices~$V_0$.
 If~$V_0$ were empty, then the sparse blow-up lemma could be directly
 applied to the partition provided by
 Lemma~\ref{lem:resil_G}. However,~$V_0$ contains, for example, those
 vertices~$v \in V(\Gamma)$ for which~$N_\Gamma(v)$ does not satisfy
 the required regularity inheritance conditions,
 and such vertices may exist.
 Hence, a significant effort in the proof of Theorem~\ref{thm:resil_main} is spent on handling~$V_0$.
 As mentioned earlier, this is achieved by pre-embedding some vertices~$X_0$ of~$H$
 onto~$V_0$. Any such pre-embedding of vertices will generate image restrictions for the remaining unembedded vertices, and so, $X_0$ needs to be chosen such that the generated \emph{image restrictions} $\{I_x\}$ and the \emph{restricting vertices} $\{J_x\}$ form a \emph{$(\rho, \zeta, \Delta, \Delta_J)$-restriction pair} as defined in Definition~\ref{defn:image_restrict}. 

To satisfy conditions~\ref{itm:IR5} and~\ref{itm:IR6} of Definition~\ref{defn:image_restrict}, we need the neighbourhood $N^*_\Gamma(J_x; V_i)$ to satisfy some good regularity properties, and as discussed above, this cannot always be met if~$J_x \subseteq V_0$. Thus, we shall find it useful to pre-embed both the vertices~$X_0$ and their neighbours~$N_H(X_0)$, while ensuring that $N_H(X_0)$ is embedded such that the regularity conditions required for the resulting image restrictions are satisfied. Here, choosing~$X_0$ from those vertices of~$H$ not contained in triangles is useful, as it adds no additional constraints on the potential images of $N_H(X_0)$ in $V(G)\setminus V_0$. 

We shall use~\cite[Lemma 28]{sparse_bandwidth} (Lemma~\ref{lem:resil_common_nbhd} below) to obtain the desired embedding of $N_H(X_0)$ as follows. The vertices of~$N_H(X_0)$ will be embedded into a subset $S \subseteq V(G)$ of linear size, where~$S$ is chosen such that every vertex of~$\Gamma$ has roughly the expected number of neighbours in~$S$.   
Lemma~\ref{lem:resil_common_nbhd} will then be sequentially applied to each vertex~$x \in X_0$, along with a large set~$W \subseteq S$ and the sets~$\{A_j\}_{j \in [K]}$ that together satisfy \ref{itm:CNL1}--\ref{itm:CNL4} below. Here, $W$ represents the candidate set of vertices for embedding the vertices of $N_H(x)$, and the sets~$\{A_j\}_{j \in [K]}$ denote large subsets of the clusters in $G$ that correspond to a particular clique of the clique-factor~$R'$. The lemma then returns, for each $x \in X_0 \subseteq V(H)$, a $\Delta$-tuple of vertices in~$W$, into which the at most $\Delta$ neighbours of~$x$ will be pre-embedded. This $\Delta$-tuple satisfies the properties~\ref{itm:W1}--\ref{itm:W4}, which will imply the image restriction requirements of~\ref{itm:SBL3} for the unembedded graph. Lemma~\ref{lem:resil_common_nbhd} is stated as follows. Note that formally it returns a set of $\Delta$ vertices in $W$ (the order is immaterial for the conclusion of the lemma) but it is convenient to put an arbitrary order on the set, which we do and refer to it as a tuple.

\begin{lemma}[Common neighbourhood lemma~\cite{sparse_bandwidth}]
  \label{lem:resil_common_nbhd}
  For each $d>0$, $K \geq 2$, and $\Delta \geq 2$, there exists
  $\zeta >0$ such that for every $\eps^\ast \in (0,1)$, there exists
  $\eps_0 >0$ such that for every $m\geq 1$ and every
  $0<\eps\le\eps_0$, there exists $C_{\CNL} >0$ for which the following
  holds. If $p \geq C_{\CNL} (\log n/n)^{1/\Delta}$, then
  $\Gamma \sim G(n,p)$ a.a.s.~satisfies the following.  Let $G=(V,E)$ be
  a not necessarily spanning subgraph of~$\Gamma$.  Let
  $\{A_j\}_{j\in[K]}$ and $W$ be pairwise disjoint subsets of $V$ such
  that for each $j,\,j'\in [K]$, we have
  \begin{enumerate}[label=\itmarab{CNL}]
  \item\label{itm:CNL1} $n/4mK\le |A_j|\le{4n}/{mK}$,
  \item\label{itm:CNL2} $(A_j,A_{j'})$ is $(\eps, d, p)$-lower-regular
    in~$G$, 
  \item\label{itm:CNL3} $|W|\ge 10^{-10}\,(\eps /mK)^4\,pn$, and
  \item\label{itm:CNL4} $|N_{G}(w;A_j)| \geq dp|A_j|$ for every $w \in W$. 
  \end{enumerate}
  Then there exists a tuple $(w_1, \dots, w_\Delta) \in \binom{W}{\Delta}$
  such that for each $\Lambda, \Lambda^\ast\subseteq[\Delta]$ and
  for each $j, j' \in [K]$ with $j \neq j'$, we have
  \begin{enumerate}[label=\itmarab{W}]
  \item\label{itm:W1}
    $\bigl|\bigcap_{i\in \Lambda} N_{G}(w_i;A_j)\bigr|\geq \zeta\,
    p^{|\Lambda|}|A_j|$,
  \item\label{itm:W2} $\bigl|\bigcap_{i\in \Lambda} N_{\Gamma}(w_i)\bigr| \le (1
    + \eps^\ast)p^{|\Lambda|}n$,
  \item\label{itm:W3} $\bigl|\bigcap_{i\in \Lambda} N_{\Gamma}(w_i;A_j)\bigr| =
    (1 \pm \eps^\ast)p^{|\Lambda|}|A_j|$, and 
  \item\label{itm:W4}
    the pair $\bigl(\,\bigcap_{i\in \Lambda}N_\Gamma(w_i;A_j),\bigcap_{i^\ast\in
      \Lambda^\ast}N_\Gamma(w_{i^\ast};A_{j'})\bigr)$ is
    $(\eps^\ast,d,p)$-lower-reg\-ular in~$G$ if we have $|\Lambda|,|\Lambda^\ast| < \Delta$ and additionally, if
    either $\Lambda\cap\Lambda^\ast=\emptyset$ or $\Delta\geq 3$ or both.
  \end{enumerate} 
\end{lemma}

After pre-embedding the vertices of~$X_0$ and their neighbours, let~$H'$ and~$G'$ denote the subgraphs induced by the unembedded vertices of~$H$ and~$G$, respectively. In order to obtain a copy of~$H$ in~$G$, it suffices to show that the graph~$H'$ can be embedded into~$G'$ in accordance with the image restrictions imposed by the pre-embedding. This will be done by an application of the sparse blow-up lemma to the graphs~$G'$ and~$H'$.
To satisfy the requirement~\ref{itm:SBL2} of Lemma~\ref{lem:rg_blowup}, we use Lemma~\ref{lem:sparse_reg_robust} to show that the regularity and inheritance properties of~$G$, obtained via Lemma~\ref{lem:resil_G}, are inherited by the graph~$G'$ as well.
Further, our application of Lemma~\ref{lem:resil_common_nbhd} when embedding $X_0\cup N_H(X_0)$ is tailored towards ensuring that the conditions for image-restricted vertices are satisfied, which will imply condition~\ref{itm:SBL3}. 

Hence, it only remains to construct the required partition of~$H'$ and the buffer vertices in accordance with condition~\ref{itm:SBL1}. 
For this we use the following lemma proved in our companion paper on the robustness of the Sauer--Spencer theorem~\cite[Lemma 6.3]{AllBoeKohNev:robust}.
This is the only place in our proof where we use the extension
property of embeddings of graphs with maximum degree at most~$\Delta$
into graphs of minimum relative degree exceeding~$\exth$ (see
Definition~\ref{defn:extension_threshold}).

 \begin{lemma}[Lemma for~$H'$~\cite{AllBoeKohNev:robust}]
    \label{lem:Lemma_H}
    For all $\Delta \geq 2$, $\gamma>0$, and $\kappa > 1$, there
    exists a constant $\alpha > 0$, such that for all
    $K \geq \Delta+1$ the following holds.  
    
    Given an $n'$-vertex graph~$G'$, let~$R$ be an $r$-vertex graph with
    ${\delta(R) \geq (\exth+ \gamma)r}$ containing a $K$-clique
    factor~$R'$.  Let $\mathcal{V} = \{V_i\}_{i \in [r]}$ be a
    partition of~$V(G')$ with parts of size
      $(1 - \gamma/2)\,n'\!/{r} \leq |V_i| \leq
      \kappa\, n'\!/r$
    for all~$i \in [r]$.  Let~$H'$ be an $n'$-vertex graph with $\Delta(H') \leq \Delta$.
    
    Further, suppose that
    $\mathcal{N} = \{N_i\}_{i \in [r]}$ is a collection of
    pairwise disjoint subsets of~$V(H')$ of size
    $|N_i| \leq \alpha |V_i|$ and such that $(H'[N], \mathcal{N}\,)$
    is an $R$-partition, where
    $N \coloneqq \bigcup_{i \in [r]} N_i$.  Then there exists a
    partition $\mathcal{X} = \{X_{i}\}_{i \in [r]}$ of $V(H')$, and
    subsets $\tilde{X}_i \subseteq X_i$ with the following properties. 
    \begin{enumerate}[label=\itmarab{H}]
    \item \label{itm:H1} The partitions $(G',\mathcal{V}\,)$ and $(H',\mathcal{X}\,)$ are size-compatible,
    \item \label{itm:H2} the partition $(H', \mathcal{X})$ is an $R$-partition, 
    \item \label{itm:H3} the family of subsets~$\tilde{\mathcal{X}} = \bigl\{\tilde{X}_i\bigr\}_{i \in [r]}$ is an $(\alpha, R')$-buffer for $(H', \mathcal{X})$, and
    \item \label{itm:H4} for all $i \in [r]$, we have $N_i \subseteq X_i$ and $N_i \cap \tilde{X}_i = \emptyset$.
    \end{enumerate}
\end{lemma}

We remark that the small $R$-partition $\mathcal{N}$ in the lemma
above will be useful to deal with image restrictions and ensure that
the chosen buffer vertices are neither image-restricted nor adjacent to image-restricted vertices. This completes the sketch of our proof of Theorem~\ref{thm:resil_main}. 

\section{Proof of Theorem~\ref{thm:resil_main}}
\label{sec:resil_proof}

In this section we prove Theorem~\ref{thm:resil_main} by applying the blow-up lemma (Lemma~\ref{lem:rg_blowup}), the lemma for~$G$ (Lemma~\ref{lem:resil_G}), the common neighbourhood lemma (Lemma~\ref{lem:resil_common_nbhd}), and the lemma for~$H$ (Lemma~\ref{lem:Lemma_H}).

\begin{proof}
We start by fixing some constants. Let $\Delta \geq 2$ and $\gamma > 0$ be given. We may assume $\gamma < 1$. Set~$\xi = \exth$ and~$K = \Delta + 1$. Note that the example witnessing the tightness of the Bollob\'{a}s--Eldridge--Catlin Conjecture implies that $\exth \geq 1 - 1/K = \Delta/(\Delta+1)$. Let~$d_0 > 0$ be returned by Lemma~\ref{lem:resil_G} with input~$\gamma$,~$K$, and~$\xi$, and set $d = \min\{ d_0, \gamma/ 20\}$. Then, with input~$d$,~$\Delta$, and~$K$, Lemma~\ref{lem:resil_common_nbhd} returns a constant~$\zeta > 0$. For~$\Delta$,~$\gamma/2$, and~$\kappa = 16$, Lemma~\ref{lem:Lemma_H} returns a constant~$\alpha > 0$. With input~$\Delta$, $d$, $\kappa$, $\alpha$, $\zeta/4$, and with~$\Delta_{R'}= K$ and~$\Delta_J = \Delta$, Lemma~\ref{lem:rg_blowup} returns constants~$\eps_{\BL} >0$ and~$\rho > 0$. We next pass~$\eps^* = 0.1\,\eps_{\BL}$ to Lemma~\ref{lem:resil_common_nbhd}, which returns a constant~$\eps_0$. We set
\[\eps \coloneqq \frac{\alpha\,\gamma\,\rho\,\zeta\, \eps^*\, \eps_0\, \eps_{\BL}^2\sqrt{d}}{200\, K\, \Delta^4}\,.\]
For this choice of~$\eps$, Lemma~\ref{lem:resil_G} now returns constants $r_1 \in \mathbb{N}$ and $C_{\LG}> 0$. 
Next, for every integer $m\le r_1/K$,
let $C_{\CNL}^{(m)} > 0$ be returned by Lemma~\ref{lem:resil_common_nbhd} for input~$m$ and~$5\eps$, and let~$C_{\CNL}$ be the maximum among all these~$C_{\CNL}^{(m)}$.
Let~$C_{\BL}$ be returned by Lemma~\ref{lem:rg_blowup} with input~$r_1$, and let~$C_{\RG}$ be returned by Lemma~\ref{lem:rand_G_property} with input~$\eps$. 
Finally, define the constant $C$, required by Theorem~\ref{thm:resil_main}, by setting
\[C^* \coloneqq \max{\{2, C_{\LG}, C_{\CNL}, C_{\BL}, C_{\RG}\}}\, \text{ and } C \coloneqq 60\, \eps^{-7} \Delta^{10}\,r_1^5\, C^*\,.\] 
\indent Next, let $\Gamma \sim G(n,p)$ with $p \geq C(\log n/n)^{1/\Delta}$ be given and assume that~$\Gamma$ satisfies the conclusions of Lemma~\ref{lem:rand_G_property}, Lemma~\ref{lem:rg_blowup}, Lemma~\ref{lem:resil_G}, and Lemma~\ref{lem:resil_common_nbhd} (for all values $m \in \mathbb{N}$ with $m \leq r_1/K$). 
Moreover, as every vertex of~$V(\Gamma)$ has $(n-1)p$~neighbours in expectation, we further assume that~$\Gamma$ has maximum degree at most~$(1+\eps)np$. The random graph~$\Gamma$ has this property with high probability, as seen by using concentration bounds (for example, \cite[Corollary 2.3]{Random_graphs_textbook_janson}) and the choice of $p$ and $C$.
Hence, we have that asymptotically almost surely, $\Gamma \sim G(n,p)$ satisfies the conclusions of all these lemmas and has $\Delta(\Gamma) \leq (1+\eps)np$. It is useful to note that~$p$ satisfies the following for $\Delta \geq 2$:
\begin{equation}
    np^2 \geq np^\Delta \geq C^\Delta\log n \geq C \log n \quad \text{ and so,} \quad C\, p^{-1}\log n \leq pn.\label{eq:p_ineq}
\end{equation}
\indent Let~$H$ and~$G\subset\Gamma$ be two $n$-vertex graphs which satisfy the assumptions of Theorem~\ref{thm:resil_main}. Our aim is to embed~$H$ into~$G$. We start by constructing a suitable partition of~$G$ by applying Lemma~\ref{lem:resil_G} to~$G$. This returns a reduced graph~$R$ on $r \coloneqq mK$ vertices for some integer $m \leq r_1/K$, 
a $K$-clique factor~$R'$ in~$R$, an exceptional set $V_0$ of size at most~$C_{\LG}\,p^{-2}$ and a $K$-equitable vertex partition $\mathcal{V} = \{V_{ij}\}_{i\in [m], j\in [K]}$ of $V(G) \setminus V_0$ such that \ref{itm:RG1}--\ref{itm:RG4} are satisfied for the constant~$d$ chosen above.

Next, we will construct a set~$X_0$ of vertices of~$H$ which will be
pre-embedded to cover the exceptional set~$V_0$. We will also
pre-embed~$N_H(X_0)$ into~$G$; our goal when embedding these vertices
is that the image restrictions this creates for $N_H\bigl(N_H(X_0)\bigr)$
satisfy the conditions required for the application of the sparse
blow-up lemma. For this, we first choose a set~$S$, with some desirable properties, into which
$N_H(X_0)$ will be pre-embedded.  

\subsubsection*{Choosing an Evenly Scattered Set $S \subseteq V(G)$}

For every~$\ell \in [\Delta]$ and every set of distinct vertices $v_1, v_2, \dots, v_{\ell} \in V(G)$, we use $\mathbf{v}_{\ell}$ to denote the $\ell$-tuple~$(v_1, v_2, \dots, v_{\ell})$. Set $\mu \coloneqq \eps^2/r$, where $r \leq r_1$ is the constant returned by the application of Lemma~\ref{lem:resil_G} above. Then, we find a subset $S \subseteq V(G)$ with the following properties. Formally, we should write $|S|=\lceil\mu n\rceil$ in what follows; we drop the ceiling here and in the rest of the paper, and observe that there is sufficient slack in all our calculations.

\begin{claim}
    \label{clm:S}
    There exists a subset~$S \subseteq V(G)$ of size $|S| = \mu n$
    such that for every $\ell \in [\Delta]$, every $\ell$-tuple of
    vertices $\mathbf{v}_\ell$, and for each graph $ F \in \{G, \Gamma\}$, we have
    \begin{equation} \label{eq:S_prop}
      \bigl|N^*_F(\mathbf{v}_\ell; S)\bigr| = \mu \bigl|N^*_F(\mathbf{v}_\ell)\bigr| \pm \eps \mu \bigl(\,|N^*_F(\mathbf{v}_\ell)| + C \log n\bigr)\,.
    \end{equation}       
\end{claim}
\begin{claimproof}
    Let $\mathcal{T}$ denote the collection of common neighbourhoods $N^*_F(\mathbf{v}_\ell)$ for every $\ell \in [\Delta]$, every $\ell$-tuple of vertices $\mathbf{v}_\ell$, and each graph $F \in \{G, \Gamma\}$. Note that we have $|\mathcal{T}| \leq 5n^{\Delta}$. Let $S \subseteq V(G)$ be chosen uniformly at random from the collection of $\mu n$-sized subsets of~$V(G)$. 
    
    For any neighbourhood~$T \in \mathcal{T}$, denoted by~$T = N^*_F(\mathbf{v}_\ell)$, note that~$|N^*_F(\mathbf{v}_\ell; S)| = |T\cap S|$ is hypergeometrically distributed with parameters~$\bigl(|V(G)|, |T|, |S| = \mu n\bigr)$. Hence, we have that $\mathbb{E} \bigl(\,|T \cap S|\,\bigr)=\mu|T|$, and so by Lemma~\ref{lem:chern_hypgeom} and as $C > 30 \Delta / \eps^3 \mu$, the following holds for each neighbourhood $T \in \mathcal{T}$:
    \[\mathbb{P} \Bigl[ |T \cap S| \neq \mu |T| \pm  \eps \mu \bigl(|T| +  C \log n\bigr) \Bigr] \leq 2e^{-\eps^2 \eps \mu C \log n/3} < {2}/{n^{1 + \Delta}}\,.\]
    Taking a union bound over all neighbourhoods $T \in \mathcal{T}$, the probability of failure is at most~$10/n$. Thus, for large enough $n$, there exists a set $S \subseteq V(G)$ that satisfies Claim~\ref{clm:S}.  
\end{claimproof}

Let~$S \subset V(G)$ be a set of vertices having properties asserted by Claim~\ref{clm:S}. When $\ell = 1$, Claim~\ref{clm:S} returns some useful bounds on the size of the neighbourhoods of a vertex, as follows. By the minimum degree condition on $G$, every~$v \in V(G)$ satisfies~$|N_G(v)| \geq pn/2$, and hence using equations~\eqref{eq:p_ineq} and~\eqref{eq:S_prop}, we have 
\begin{equation}
    |N_G(v; S)| \geq (1 - \eps)\mu \cdot pn/2\ - \eps \mu C \log n \geq (0.5 - 2\eps) \mu pn\,. \label{eq:NGvSlb}
\end{equation}
Similarly, as $\Delta(\Gamma) \leq (1+\eps)pn$, and by equation~\eqref{eq:p_ineq} and~\ref{itm:RG1}, the following holds for all $i \in [m]$, $j \in [K]$, and $v \in V(\Gamma)$:
\begin{equation}
\begin{split}
    |N_{\Gamma}(v; S \cap V_{ij})| &\leq |N_{\Gamma}(v; S)| \leq (1 + \eps)\mu \cdot (1 + \eps)pn + \eps \mu C \log n\\ 
    &\leq 10\, \mu pn = 40\,\eps^2p (n/4r) \leq 40\,\eps^2p|V_{ij}| \leq  \eps p |V_{ij}|\,.
\end{split} \label{eq:s_vij}
\end{equation}

\subsubsection*{The Pre-embedding Algorithm}

By fixing the set~$S$ as given by Claim~\ref{clm:S}, we are now ready to carry out the pre-embedding process. Let~$Z \subseteq V(H)$ be the set of vertices of~$H$ that are not in any triangles in~$H$. By the assumptions of Theorem~\ref{thm:resil_main}, we have $|Z| \geq C p^{-2}$. We use the following pre-embedding algorithm that covers the exceptional set~$V_0$ using vertices~$X_0 \subset Z$ and possibly some vertices of $N(X_0)$, and further also embeds~$N_H(X_0)$.

\begin{algorithm}
\caption{Pre-embedding}
\label{alg:pre}
 Set $t \coloneqq 0$ and $\varphi_0 \coloneqq \emptyset$ \;
 \While{$V_0\setminus\im(\varphi_t)\neq\emptyset$}{
  $t \coloneqq t+1$ \;
  Choose $v_{t} = v\in V_0\setminus\im(\varphi_{t-1})$ so that $\bigl|N_G(v; S)\setminus\im(\varphi_{t-1})\bigr|$ is minimised
  \nllabel{line:choosev} \;

  Choose $x_{t}\in Z$ so that $\dist_H\bigl(x_{t}, \{x_1, \dots, x_{t-1}\}\bigr) \geq 8$ 
  \nllabel{line:x} \;
  
  Let $\{y_1,\dots,y_\ell\}=N_H(x_{t})$ \nllabel{line:y} \;
  
  Choose distinct $w_1,\dots,w_{\ell}\in N_G\bigl(v_{t}; S\setminus\im(\varphi_{t-1})\bigr)$ that satisfy \ref{itm:W1}--\ref{itm:W4}
  \nllabel{line:choosenbs} \;
  
  $\varphi_{t}\coloneqq \varphi_{t-1}\cup\{x_{t}\mapsto v_{t}\}\cup\{y_1\mapsto w_1\}\cup\dots\cup\{y_\ell\mapsto w_\ell\}$ \;
 }
\end{algorithm}

Let $\varphi_{t^*}$ denote the partial embedding obtained at the end of the pre-embedding algorithm, and let $X_0 = \{x_1, \dots, x_{t^*}\} \subseteq Z$ be the set of vertices of~$Z$ embedded to~$V_0$ in line~\ref{line:x} of the algorithm. For the pre-embedding algorithm, it remains for us to show that it is always possible to choose a vertex~$x_{t}$ as required in line~\ref{line:x}, and to describe how the vertices $w_1, \dots, w_\ell \in N_G\bigl(v_{t}; S\setminus\im(\varphi_{t-1})\bigr)$ are chosen in line~\ref{line:choosenbs}. Observe that as $x_t$ lies in no triangles in $H$, there are no adjacency conditions imposed on the choice of vertices~$w_1, \dots, w_{\ell}$.

For the former, let~$\tilde{H}$ denote the auxiliary graph with vertex set~$V(H)$, and edges~$xy \in E(\tilde{H})$ if and only if~$\dist_H(x,y) \leq 7$. Now, consider the induced subgraph $\tilde{H}[Z]$. The maximum degree of~$\tilde{H}[Z]$ is at most~$\Delta(\tilde{H}[Z]) \leq \Delta^{\!8} -1$. By the Hajnal--Szemer\'edi Theorem (Theorem~\ref{thm:Hajnal-Szemeredi}) applied to~$\tilde{H}[Z]$, there exist at least~$|Z|/\Delta^{\!8} \geq Cp^{-2}/\Delta^{\!8}$ vertices in~$Z$ that form an independent set in $\tilde{H}[Z]$, and thus have pairwise distance at least~$8$ in~$H$. Since $|V_0| \leq C_{\LG}\,p^{-2}$ and $C > \Delta^{\!8}\, C_{\LG}$, this subset of vertices in~$Z$ is sufficiently large to be pre-embedded to the entirety of~$V_0$ under the constraints of line~\ref{line:x}, as required. 

Observe that as every pair of vertices in~$X_0$ has distance at least $8$ in the graph~$H$, the set~$X_0$ cannot contain the vertices $y_1, \dots, y_\ell \in N_H(x_t)$ mentioned in line~\ref{line:y} for any time $t \leq t^*$. Next, we need to describe how the vertices $w_1,\dots,w_\ell \in N_G\bigl(v_{t}; S\setminus\im(\varphi_{t-1})\bigr)$ will be chosen in line~\ref{line:choosenbs}. As these vertices are used to pre-embed~$N_H(x_t)$, we require that~$w_1, \dots, w_{\ell}$ are chosen such that they satisfy the properties~\ref{itm:W1}--\ref{itm:W4}. This will in turn aid in proving the image restriction condition~\ref{itm:SBL3} required for embedding the remaining graph using Lemma~\ref{lem:rg_blowup} later in the proof. To meet these requirements, the vertices~$w_1, \dots, w_\ell$ shall be chosen by an application of Lemma~\ref{lem:resil_common_nbhd}. For this, we first show that the set~$C_t \coloneqq N_G\bigl(v_{t}; S\setminus\im(\varphi_{t-1})\bigr)$ stays large enough at all times $t \leq t^*$. 

\begin{claim}
    \label{clm:avgS}
    At each time $t \leq t^*$, the set $C_t \coloneqq N_G\bigl(v_{t}; S\setminus\im(\varphi_{t-1})\bigr) $ has at least $\mu pn/4$ vertices. 
\end{claim}
\begin{claimproof}
  Note that for any time $t \leq t^*$, as each vertex $x_t \in X_0$ has at most $\Delta$ neighbours in $H$, we have $|\im(\varphi_{t-1})|\leq (\Delta+1)t$. Thus, if $t \leq 0.2 \mu p n/(\Delta+1)$, then by equation~\eqref{eq:NGvSlb}, we have 
  \[
    |C_t| \geq |N_G(v_t; S)| - |\im(\varphi_{t-1})| \geq (0.3 -
    2\eps)\mu p n \geq \mu pn/4,
  \]
  as required. Hence, we may assume
  that $t >  0.2 \mu p n/ (\Delta+1)$. Suppose for a contradiction
  that $|C_t| < \mu pn/4$. Since at each time~$t'$, the vertex~$v_{t'}$
  is chosen to minimise~$|C_{t'}|$, it follows that for every $t' \in
  \bigl( t -  \mu pn/20(\Delta+1), t\bigr)$, the set~$C_{t'}$ is
  such that
  \[
    \bigl|C_{t'}\bigr| \leq \bigl|N_G(v_t; S \setminus \im(\varphi_{t'
      - 1}))\bigr| \leq \bigl|C_t\bigr| + \bigl|\im(\varphi_{t-1})
    \setminus \im(\varphi_{t'-1})\bigr| \leq (1/4 + 1/20)\mu pn = 0.3
    \mu p n.
  \]
  
  Let $B$ be the set of all vertices $v_{t'}$ chosen over this time interval, and let $A$ denote the set $\im(\varphi_{t-1}) \subseteq V(\Gamma)$. By equation~\eqref{eq:p_ineq} and as $C > (40\Delta/\eps\mu)\max\{C_{\RG}, C_{\LG}\}$ (where we use~$\mu = \eps^2/r \geq \eps^2/r_1)$,  we obtain the following bounds on the size of~$A$.
    \begin{align*}
    |A| &\geq t-1 \geq 0.15 \mu p n/(\Delta +1)\geq 0.1 \mu\ C p^{-1}\log n /(\Delta+1)
    > C_{\RG}\,p^{-1} \log n, \quad\text{and}\\[3pt]
    |A| &\leq (\Delta+1)|V_0| \leq (\Delta+1) \,C_{\LG}p^{-2}  \leq  \eps \mu\,  C p^{-2} \leq \eps \mu\, n/\log n \leq  \eps \mu n. 
    \end{align*}
    As the random graph~$\Gamma$ satisfies the outcome of Lemma~\ref{lem:rand_G_property} and as $|A| \geq C_{\RG}\, p^{-1} \log n$, it follows that at most $C_{\RG}\, p^{-1} \log (en/|A|)$ vertices $v \in V(\Gamma)$ satisfy the inequality $|N_{\Gamma}(v; A)| > (1+\eps)p|A|$. We show that this is not the case, thereby contradicting the assumption that $|C_t| < \mu p n /4$. 
    
    For this, consider some vertex $b\in B$, where $b = v_{t'}$. By equation~\eqref{eq:NGvSlb}, we have that $|N_G(b; S)| \geq (0.5 - 2\eps) \mu p n$, and by choice of $t'$, we have that $|C_{t'}| = |N_G(b;S)\setminus \im(\varphi_{t'-1})|  \leq 0.3 \mu p n$.  Hence for the set $A = \im(\varphi_{t-1})$, we obtain the inequality~$|N_G(b; S \cap A)| \geq (0.2-2\eps)\mu pn \geq 0.1 \mu p n $. Thus, as~$|A| \leq \eps \mu n$, for every vertex~$b \in B$, the size of the set~$N_\Gamma(b; A)$ is bounded below by 
    \[|N_{\Gamma}(b; A)| \geq |N_\Gamma(b; S \cap A)| \geq 0.1 \mu p n  \geq 0.1 \eps^{-1}p |A| > (1+\eps)p |A|.\] 
    Moreover, as $|A|$ is bounded below by $\log n$, for sufficiently large $n$, the set~$B$ has size at least 
    \[|B| \geq \mu p n/20(\Delta+1)-1 \geq (\mu/40\Delta) C p^{-1} \log n \geq  C_{\RG}\, p^{-1}\log\bigl(en/|A|\bigr)\,.\] 
    Thus the outcome of Lemma~\ref{lem:rand_G_property} is violated for the set~$A$, as witnessed by the vertices~$b \in B$. Hence, by contradiction, we have that $|C_t| \geq \mu p n/4$ for all times $t$, as required.
\end{claimproof} 

Thus, at each time $t$ of the algorithm, there are enough available vertices in $C_t$ to pick the images $w_1, \dots, w_{\ell}$ of $y_1, \dots, y_{\ell}$. To choose these vertices, we use Lemma~\ref{lem:resil_common_nbhd}, for which we first find a subset~$W_t \subseteq C_t$ that satisfies the properties~\ref{itm:CNL3} and~\ref{itm:CNL4}. 

\begin{claim}
\label{clm:CNL34}
    For each time $t \leq t^*$, there exist a subset $W_t \subseteq C_t$ with size~$|W_t| \geq \mu p n/8m$, and an index~$i_t \in [m]$ such that for all $w \in W_t$ and $j \in [K]$, we have $|N_G(w; V_{i_t\,j})| \geq 2d p |V_{i_t\,j}|$.   
\end{claim}
\begin{claimproof}
    Given the set $C_t$ at time $t$, let $Y_0 \subseteq C_t$ be the set of vertices $Y_0 \coloneqq \{y \in C_t : |N_{\Gamma}(y; V_0)| > \eps pn\}$. We first obtain an upper bound on the size of $Y_0$ using Lemma~\ref{lem:rand_G_property}. Note that if $Y_0 \neq \emptyset$, then by equation~\eqref{eq:p_ineq} and as $C \geq 2\eps^{-1}C_{\RG}$, we have $|V_0| \geq \eps pn \geq \eps\, C p^{-1}\log n \geq C_{\RG} p^{-1} \log n$. 
    Similarly, as $|V_0| \leq C_{\LG}p^{-2}$, every vertex~$y \in Y_0$ has at least \[|N_{\Gamma}(y; V_0)| \geq \varepsilon p n \geq \eps\, C p^{-1} \log n \geq  2 C_{\LG}p^{-2}\cdot p > (1+\eps)p|V_0|\] $\Gamma$-neighbours in $V_0$. Thus, by Lemma~\ref{lem:rand_G_property}, $Y_0$ has size at most $|Y_0| \leq C_{\RG}\, p^{-1} \log\bigl(en/|V_0|\bigr)$. 

    Analogously, for all $i \in [m]$ and $j \in [K]$, 
    define $Y_{ij} \coloneqq \big\{y \in C_t: |N_{\Gamma}(y; V_{ij})| > (1+\eps)p|V_{ij}| \big\}$. Note that by~\ref{itm:RG1} and as $r \leq r_1$, we have $|V_{ij}| \geq n/4r \geq n/4r_1$. It then follows from equation~\eqref{eq:p_ineq} and the relation~$C \geq (4r_1/\eps)\, C_{\RG}$, that $|V_{ij}| \geq C_{\RG}\,p^{-1} \log n$. Hence, by the conclusion of Lemma~\ref{lem:rand_G_property} applied to the set~$V_{ij}$, we have that~$|Y_{ij}| \leq C_{\RG}\, p^{-1} \log \bigl(en/|V_{ij}|\bigr)$ for all $i \in [m]$  and $j \in [K]$. Now, if we let $Y \coloneqq Y_0 \cup \bigl( \bigcup_{\,i \in [m], j \in [K]} Y_{ij} \bigr)$, then $|Y| \leq (r+1)C_{\RG}\, p^{-1} \log n \leq \mu p n/8$, where the latter inequality follows from equation~\eqref{eq:p_ineq} and the relation $C \geq (16\,r_1^2/\eps^2)\, C_{\RG} \geq (16\,r_1/\mu)\, C_{\RG}$.
    
    Set $W' = C_t \setminus Y$, and fix some $w \in W'$.
    For a contradiction, suppose that for every $i \in [m]$, there exists some $j \in [K]$, such that $|N_G(w; V_{ij})| \leq 2dp|V_{ij}|$. Then, by the $K$-equitability of the partition~$\mathcal{V} = \{V_{ij}\}$ of $V(G) \setminus V_0$, we have
    \[|N_G(w)| \leq \eps pn + 2d p \frac{n}{K} + (1+ \eps)(K-1)p \frac{n}{K} + 2r < \biggl( \frac{K-1}{K} + \frac{\gamma}{2} \biggr) pn.\]
    This contradicts the minimum degree condition for $G$, as $\exth \geq \Delta/(\Delta+1)$. Thus, for each $w \in W'$, there exists an index $i(w) \in [m]$ such that for all $j \in [K]$, we have~$|N_G(w; V_{i(w)j})| \geq 2dp|V_{i(w)j}|$. Letting $i_t \in [m]$ be the modal index over all $w \in W'$ and setting $W_t = \{w \in W': i(w) = i_t\}$ gives the required set $W_t \subseteq C_t$, with size~$|W_t| \geq |W'|/m \geq \mu p n /8m$. 
\end{claimproof}

As $\Gamma$ satisfies the conclusion of Lemma~\ref{lem:resil_common_nbhd} with input $d$, $K$, $\Delta$, $\eps^*= 0.1 \eps_{\BL}$, $m$, and $5\eps$, we apply this good event to choose the images~$w_1, \dots, w_{\ell}$ of~$y_1, \dots, y_{\ell}$ at each time~$t \leq t^*$ in step~\ref{line:choosenbs} of the pre-embedding algorithm. Given some time~$t \leq t^*$, let~$W_t \subseteq C_t$ and index~$i_t \in [m]$ be returned by Claim~\ref{clm:CNL34}, and let $W \subseteq W_t$
be any $\mu p n/8m$-sized subset of $W_t$. For each $j \in [K]$, set $A^t_j \coloneqq V_{i_t\,j}\setminus W$. Then,~$\{A_j^t\}_{j \in [K]}$ and~$W$ are pairwise disjoint subsets of~$V(G)$. Moreover, by~\ref{itm:RG1}, we have 
\[
  4n/mK \geq |A^t_j| \geq (1-2\eps)n/r - \mu p n/8m \geq (1 - 4\eps)
  n/r \geq n/4mK.
  \] 
This is~\ref{itm:CNL1} in Lemma~\ref{lem:resil_common_nbhd}.
Next, observe that by~\ref{itm:RG1}, we have~$|W| \leq \eps^2pn/4r \leq \eps^2 p|V_{ij}|$ for all~$i \in [m]$ and~$j \in [K]$, and 
hence, \begin{equation}|A^t_j \symmdiff V_{i_t\,j}| = |V_{i_t\,j} \setminus A_j^t| \leq |W| \leq \eps^2 |V_{i_t\,j}|. \label{eq:Ajt}\end{equation} 
Then, as $(V_{i_t\,j}, V_{i_t\,j'})$ forms an $(\eps, d, p)$-lower-regular pair by~\ref{itm:RG2}, it follows from Lemma~\ref{lem:sparse_reg_robust} 
that $(A^t_j, A^t_{j'})$ is $(5\eps, d, p)$-lower-regular for all~$j \neq  j' \in [K]$. This proves~\ref{itm:CNL2}. Finally,~\ref{itm:CNL3} follows from the choice of~$\mu$, and~\ref{itm:CNL4} holds by Claim~\ref{clm:CNL34} and the fact that
\[
  |N_G(w; A^t_j)| \geq 2dp|V_{i_t\,j}| - |W| \geq
  (2d-\eps^2)p|V_{i_t\,j}| \geq dp|A^t_j|.
\] 
\indent Thus, by Lemma~\ref{lem:resil_common_nbhd}, there exist vertices $w_1, \dots, w_{\Delta} \in W \subseteq W_t$ which satisfy~\ref{itm:W1}--\ref{itm:W4} for the sets $A^t_j$. We complete the pre-embedding algorithm by assigning the vertices $y_1, \dots, y_\ell$ arbitrarily to the vertices~$w_1, \dots, w_\ell$, where $\ell \leq \Delta$. 

 \subsubsection*{Image Restrictions and Buffer Vertices}

Recall that $X_0$ was the set of vertices of $H$ embedded to $V_0$ in line~\ref{line:x} of the pre-embedding algorithm, and that the algorithm embeds $N_H[X_0]$ into $V_0 \cup S$. Set $V'_{ij} \coloneqq V_{ij} \setminus \im(\varphi_{t^*})$, and similarly define the set of available vertices $V' \coloneqq  \bigcup_{i,j} V'_{ij} $, the induced subgraph $G' \coloneqq G[V']$, and the partition $\mathcal{V}' \coloneqq \{V'_{ij}\}_{i \in [m], j\in[K]}$ of $V(G')$. We first obtain a lower bound on the sizes of $V'_{ij}$. Note that by equation~\eqref{eq:p_ineq} and as $C > (40\Delta\, r_1/\eps^3)\,C_{\LG} \geq (40\Delta/\eps\mu)\,C_{\LG}$, we have
\begin{equation}
       \bigl|\im(\varphi_{t^*})\bigr| \leq (\Delta+1)|V_0| \leq (\Delta+1) \,C_{\LG}p^{-2}  \leq  \eps \mu\,  C p^{-2} \leq \eps \mu n/\log n \leq  \eps \mu n.
       \label{eq:V0_ub}
\end{equation}
Then, as $|V'_{ij}| \geq |V_{ij}| - |\im(\varphi_{t^*})|$ and using~\ref{itm:RG1}, we obtain the following lower bounds on the sizes of the vertex sets~$V'_{ij}$ and~$V(G')$:
\begin{equation}
    |V_{ij}'| \geq (1 - 4r\mu\eps)|V_{ij}| \geq (1- \eps^2)|V_{ij}| \quad \text{and similarly,}\quad |V(G')| \geq (1 - \eps^2)n . \label{eq:Vprime_lb}
\end{equation} 

Let $X' \coloneqq V(H) \setminus \dom(\varphi_{t^*})$ be the set of unembedded vertices in~$H$, and let~$H' \coloneqq H[X']$ be the subgraph of~$H$ induced by~$X'$. We will use Lemma~\ref{lem:rg_blowup} to embed~$H'$ into~$G'$ in a manner that extends the pre-embedding~$\varphi_{t^*}$ to an embedding of~$H$ into~$G$. For this, we first need to construct a size-compatible partition~$\mathcal{X}' \coloneqq \{X'_{ij}\}_{i \in [m], j\in[K]}$ of $V(H')$ that satisfies the requirements for image restrictions and buffer vertices in Lemma~\ref{lem:rg_blowup}.

Let~$X_I$ be the set of vertices in~$V(H')$ which have at least one neighbour in~$N_H(X_0)$. This will form the set of image-restricted vertices in~$X'$.  Then, for each~$x \in X_I$, there is a unique vertex~$x_t \in X_0$ such that $N_H(x) \cap N_H(x_t) \neq \emptyset$. Indeed, for otherwise, if~$N_H(x)$ intersects~$N_H(x_t)$ and~$N_H(x_{t'})$, then~$\dist_H(x_t, x_{t'}) \leq 4$, which contradicts the choice of $x_t$ and $x_{t'}$ in line~\ref{line:x} of the pre-embedding algorithm. For each time~$t \leq t^*$, let~$X_I^t$ be the vertices of~$X_I$ for which~$N_H(x) \cap N_H(x_t) \neq \emptyset$, and let $N_t \coloneqq N_{H'}[X_I^t] = N_{H'}(X_I^t) \cup X_I^t$ be the closed neighbourhood of~$X_I^t$ in~$H'$. 

Note that for $t \neq t'$, there is no edge between $N_t$ and $N_{t'}$, for otherwise, $\dist_H(x_t, x_{t'}) \leq 7$, again contradicting the choice of~$x_t$ and~$x_{t'}$. For each $t \leq t^*$, we have $\Delta\big(H[N_t]\big) \leq \Delta \leq K-1$, and so by Theorem~\ref{thm:Hajnal-Szemeredi}, the vertex set $N_t$ can be equitably partitioned into $K$ independent sets. Add each of these independent sets arbitrarily to distinct parts in $\bigl\{X'_{i_t\,j}\bigr\}_{j \in [K]}$, where $i_t\in [m]$ is the index returned by Claim~\ref{clm:CNL34} at time~$t$. Repeat this for each time~$t \leq t^*$. Note that this step adds at most~$\eps|V'_{ij}|$ vertices to the part~$X'_{ij}$, as seen by the inequality
\begin{equation}\begin{split}
    \bigl|N_{H'}[X_I]\bigr| &\leq (\Delta+1)|X_I| \leq \Delta(\Delta+1) \bigl|\dom(\varphi_{t^*})\bigr| 
    \overset{\eqref{eq:V0_ub}}{\leq} \eps \mu n 
    \overset{\text{\ref{itm:RG1}}}{\leq} \eps \mu\, 4r|V_{ij}|\\
    &= 4 \eps^3 \,|V_{ij}| \overset{\eqref{eq:Vprime_lb}}{\leq}  \frac{4\eps^3}{1-\eps^2} |V_{ij}'| \leq \eps |V_{ij}'| \leq \min\bigl\{\rho |V'_{ij}|, \alpha|V'_{ij}| \bigr\} \,. \label{eq:NXI}
\end{split}
\end{equation}
Here, the third inequality follows from a computation similar to equation~\eqref{eq:V0_ub}, but with the relation $C \geq (40\Delta^3/\eps\mu)\,C_{\LG}$ instead. We now partition the remaining vertices of $V(H')$ (i.e. the vertices of~$H'$ not in~$N_{H'}[X_I]$) into the parts~$\{X'_{ij}\}$, and obtain potential buffer vertices using Lemma~\ref{lem:Lemma_H}. 

Note that by Lemma~\ref{lem:resil_G}, the reduced graph $R$ has minimum degree $\delta(R) \geq (\exth + \gamma/2\,)r$ and $R'$ is a $K$-clique factor of $R$. Moreover, by~\ref{itm:RG1}, equation~\eqref{eq:Vprime_lb}, and as $\kappa = 16$ and $\eps \leq \gamma/24$, the following holds for the parts~$V'_{ij}$ for all $i \in [m]$ and $j \in [K]$:
\begin{align}
    \begin{split}
    |V'_{ij}| &\leq 4n/r \leq 8(1-\eps^2)\,n/r \leq 8 |V(G')|/r\leq \kappa\, |V(G')|/r \quad\text{and}\\[2pt]
    |V'_{ij}| &\geq (1-\eps^2) |V_{ij}| \geq (1-3\eps)\, n/r \geq \bigl(1 - \gamma/4\bigr)|V(G')|/{r}.
    \end{split}\label{eq:VpwrtG}
\end{align}
\indent Finally, for each~$i \in [m]$ and~$j\in [K]$, let~$N_{ij}$ be the set of vertices of~$N = N_{H'}[X_I]$ added to the part~$X'_{ij}$, and set~$\mathcal{N} \coloneqq \{N_{ij}\}_{i\in[m], j\in[K]}$. This partition $(H'[N], \mathcal{N})$ is an $R$-partition by construction. Moreover, by equation~\eqref{eq:NXI}, we have that $|N_{ij}| \leq \alpha|V'_{ij}|$ for all $i \in [m]$ and $j\in [K]$. Thus, Lemma~\ref{lem:Lemma_H} with input $\Delta$, $\gamma/2$, $\kappa = 16$, and $K = \Delta+1$, returns a partition $\mathcal{X}' = \{X'_{ij}\}$ of $V(H')$ and a collection of subsets $\tilde{X}_{ij} \subseteq X'_{ij}$, for which~\ref{itm:H1}--\ref{itm:H4} hold and such that~$N_{ij} \subseteq X'_{ij}\setminus \tilde{X}_{ij}$.     

Let $X_{ij}^* \coloneqq N_{ij} \cap X_I$ denote the image-restricted vertices added to the part $X'_{ij}$. 
Then, for each $x \in V(H')$, the set of restricting vertices $J_x \subseteq V(\Gamma) \setminus V(G')$ and the image restrictions $I_x \subseteq V(G')$ can be defined as follows. Set $V(x) \coloneqq V'_{ij}$ for all $x \in X'_{ij}$. Then, for all $x \in V(H')$, we have 
\[
J_x \coloneqq \varphi_{t^*}\big(N_H(x) \cap N_H(X_0) \big) \quad
\text{and}\quad I_x \coloneqq N^*_G \big(J_x ; V(x)\big).
\]
Observe that~$J_x = \emptyset$ and hence,~$I_x = V(x)$ for all non-image-restricted vertices~$x \notin X_I$.

 \subsubsection*{Applying the Sparse Blow-up Lemma}

 We are now ready to apply Lemma~\ref{lem:rg_blowup} to the graphs~$G'$ and~$H'$. Note that we used Lemma~\ref{lem:resil_G} to obtain a reduced graph~$R$ on~$r \leq r_1$ vertices and with a spanning $K$-clique factor~$R'$ of degree~$\Delta(R') \leq \Delta_{R'} = K$. For the graph~$G' \subseteq \Gamma$, we have the partition~$\mathcal{V}' = \{V'_{ij}\}$ of~$V(G')$. This partition is~$\kappa = 16$ balanced as, by equation~\eqref{eq:VpwrtG}, we have~$|V(G')|/2r \leq |V'_{ij}| \leq 8|V(G')|/r$ for all $i \in [m]$ and $j \in [K]$. Moreover, each part has size at least $n/\kappa r_1$, as required by the preamble of the good event of Lemma~\ref{lem:rg_blowup}. Finally, for the graph $H'$ with $\Delta(H') \leq \Delta(H) \leq \Delta$, we have a partition $\mathcal{X}' = \{X'_{ij}\}$ that is size-compatible to $\mathcal{V}'$ by~\ref{itm:H1}.
 Thus, it simply remains to check that the conditions~\ref{itm:SBL1}--\ref{itm:SBL3} are satisfied.

Condition~\ref{itm:SBL1} directly follows from~\ref{itm:H2} and~\ref{itm:H3}. For~\ref{itm:SBL2}, note that by equation~\eqref{eq:Vprime_lb}, we have that $|V_{ij} \symmdiff V'_{ij}| \leq
\eps^2|V_{ij}|$. Hence, by~\ref{itm:RG2} and Lemma~\ref{lem:sparse_reg_robust}, the partition~$\mathcal{V'} = \{V'_{ij}\}$ is~$(5\eps, d, p)$-regular on~$R$ and~$R'$. By choice of $\eps \leq \eps_{\BL}/5$, this implies
$(\eps_{\BL}, d, p)$-regularity. Next we verify the regularity inheritance conditions of~\ref{itm:SBL2}. It follows from equation~\eqref{eq:s_vij} and~\ref{itm:RG4} that
\begin{equation}
  \bigl|N_{\Gamma}(v; S \cap  V_{ij} )\bigr|
  \leq 40\, \eps^2 p\bigl|V_{ij}\bigr|
  \leq 40\,\eps^2/(1-\eps)\cdot\big|N_{\Gamma}(v; V_{ij})\big| \leq
  \eps\big|N_{\Gamma}(v; V_{ij})\big|\,.
  \label{eq:deg_check2}
\end{equation}
Since the pre-embedding algorithm does not embed vertices outside~$V_0 \cup S$, it follows that
 \[\big|N_{\Gamma}(v; V_{ij}) \symmdiff N_{\Gamma}(v; V'_{ij})\big| \leq \bigl|N_{\Gamma}(v; S \cap  V_{ij} )\bigr| \leq  \eps \big|N_{\Gamma}(v; V_{ij})\big|\,.\]
\indent
 Now, Lemma~\ref{lem:sparse_reg_robust} along with~\ref{itm:RG3} and the relation~$\eps_{\BL} \geq 10 \sqrt{\eps}$, implies the one-sided inheritance condition on~$R'$. The two-sided inheritance condition on~$R'$ for~$\tilde{\mathcal{X}}$ follows by additionally observing that~$\mathcal{X}'$ is an $R$-partition by~\ref{itm:H2}. Finally, we verify the minimum degree conditions for super-regularity of~$(G', \mathcal{V'})$ on~$R'$. Fix any $v \in V'_{ij'}$ and let $j \neq j'$. By~\ref{itm:RG2}, we know that 
 \[\bigl|N_G(v; V_{ij})\bigr| \geq (d-\eps) \max \bigl\{ p|V_{ij}|, |N_{\Gamma}(v; V_{ij})|/2\bigr\}\,.\]
 Moreover, by equations~\eqref{eq:s_vij} and \eqref{eq:deg_check2}, we have
 \[\bigl|N_G(v; V_{ij} \cap S )\bigr| \leq \min\bigl\{ \eps p |V_{ij}|, \eps|N_{\Gamma}(v; V_{ij})|\, \bigr\} \leq 2\eps \max\bigl\{ p |V_{ij}|, |N_{\Gamma}(v; V_{ij})|/2\bigr\}.
 \]
 Combining both the above equations, and using the fact~$V_{ij}\setminus V_{ij}' \subseteq V_{ij} \cap S$,  we obtain 
\begin{align*}
    \bigl|N_G(v; V_{ij}')\bigr| &\geq \bigl|N_G(v; V_{ij})\bigr| -\bigl|N_G(v; S \cap V_{ij})\bigr|
    \geq (d-3\eps) \max \bigl\{ p|V_{ij}|, |N_{\Gamma}(v; V_{ij})|/2 \bigr\}\\[2pt]
    &\geq (d-\eps_{\BL}) \max \bigl\{ p|V_{ij}'|, |N_{\Gamma}(v; V'_{ij})|/2 \bigr\}\,,
\end{align*}
as required for the minimum degree condition for $(\eps_{\BL}, d, p)$-super-regularity of $\mathcal{V'}$ on $R'$. This completes the verification of~\ref{itm:SBL2} for $G'$. 

Finally, to verify~\ref{itm:SBL3}, we need to show that $\mathcal{I}$ and $\mathcal{J}$ form a $(\rho, \zeta/4, \Delta, \Delta_J = \Delta)$-restriction pair (see Definition~\ref{defn:image_restrict}). Indeed,~\ref{itm:IR1} follows from the bound $|X_{ij}^*| \leq |N_{H'}[X_I]| \leq \rho |X'_{ij}|$, obtained using equation~\eqref{eq:NXI} and the size compatibility of~$\mathcal{X}'$ and~$\mathcal{V}'$. Both~\ref{itm:IR3} and~\ref{itm:IR4} are consequences of $\Delta(H) \leq \Delta$ and the definition of $J_x$. More specifically, $|J_x| + \deg_{H'}(x) = \deg_H(x) \leq \Delta$ and every vertex of $N_H(X_0)$ is adjacent to at most $\Delta$ neighbours in $H'$. 

It remains to establish~\ref{itm:IR2},~\ref{itm:IR5}, and~\ref{itm:IR6}.
For this, fix $x \in X_{ij}^*$. Recall that for~$x \in X^*_{ij}$, there exist~$t \leq t^*$ and~$x_t \in X_0$ such that~$J_{x} \subseteq \varphi_{t^*}(N_H(x_t))$, and so $J_x$ satisfies~\ref{itm:W1}--\ref{itm:W4} of Lemma~\ref{lem:resil_common_nbhd} for the sets~$A_j^t \subseteq V_{i_t\,j}$ defined after the proof of Claim~\ref{clm:CNL34}. Note that while $A^t_j$ need not be a subset of $V_{i_t\,j}'$, we have that~$A^t_j \symmdiff V'_{i_t\,j} \subseteq S \cap V_{i_t\,j}$, which will be useful to verify~\ref{itm:IR2},~\ref{itm:IR5}, and~\ref{itm:IR6}. Moreover, by equation~\eqref{eq:Ajt}, we have that
\begin{equation}
    |A^t_j| \geq (1-\eps^2)|V_{i_t\,j}| \geq (1/2)|V_{i_t\,j}|.\label{eq:Atj_Vtj}
\end{equation}

For~\ref{itm:IR2}, since $A^t_j \setminus S\subseteq V'_{i_t\,j}$, we
obtain a lower bound on the size of~$I_x = N^*_G(J_x; V'_{i_t\,j})$
using the relation $|I_x| \geq |N^*_G(J_x; A^t_j)| -
|N^*_{\Gamma}(J_x; S)|$. Using~\ref{itm:W1}, we get
\[
  |N^*_G(J_x; A^t_j)| \geq \zeta\, p^{|J_x|} |A^t_j| \geq
  (\zeta/2)\,p^{|J_x|}|V_{i_t\,j}|.
\]
On the other hand, to bound~$|N^*_{\Gamma}(J_x; S)|$, observe that by
\ref{itm:W2}, we have $|N^*_{\Gamma}(J_x)| \leq (1 +
\eps^*)\,p^{|J_x|}n$. Hence, by equation~\eqref{eq:S_prop} applied to the
at most~$\Delta$ many vertices of~$J_x$, we get that, for all $i' \in
[m]$ and $j' \in [K]$, 
\begin{align}\begin{split}
\bigl|N^*_{\Gamma}(J_x; S)\bigr| &\leq \mu(1+\eps)(1+\eps^*)\,p^{|J_x|}n + \eps\mu C \log n \\
&\leq 4 \mu p^{|J_x|}n +  \eps \mu  p^{|J_x|}n \leq 10\,\mu p^{|J_x|} n \leq \eps p^{|J_x|}|V_{i'j'}|, \label{eq:Jx_S}
\end{split}\end{align}
where the second inequality follows from equation~\eqref{eq:p_ineq},
as $|J_x| \leq \Delta$, and $\eps \leq \eps^* \leq 1$; the
last inequality uses~\ref{itm:RG1}.  Now, \ref{itm:IR2}~follows by
combining both the above inequalities along with the relation~$\eps
\leq \zeta/4$; indeed, 
\[|I_x| \geq |N^*_G(J_x; A^t_j)| - |N^*_{\Gamma}(J_x; S)| \geq (\zeta/2 - \eps)\, p^{|J_x|} |V_{i_t\,j}| \geq (\zeta/4)\,(dp)^{|J_x|}|V'_{i_t\,j}|\,.\]

Property~\ref{itm:IR5} is trivially true if $J_x = \emptyset$, and so we may assume that
$J_x \neq \emptyset$. By~\ref{itm:W3}, we have that $|N^*_{\Gamma}(J_x; A_j^t)| = (1 \pm \eps^*)\,p^{|J_x|}|A_j^t|$. Combining this with equations~\eqref{eq:Atj_Vtj} and~\eqref{eq:Jx_S}, and using $A^t_j\setminus S \subseteq V'_{i_t\,j}$, we obtain the following lower bound (recall that $\eps \leq \eps^* \leq 0.1\, \eps_{\BL}$):
\begin{align*}
    \bigl|N^*_{\Gamma}(J_x; V'_{i_t\,j})\bigr| 
    & \geq |N^*_{\Gamma}(J_x; A^t_j)| - |N^*_{\Gamma}(J_x; S)| \geq (1-\eps^*)\,p^{|J_x|}|A_j^t| - \eps p^{|J_x|}|V_{i_t\,j}| \\
    &\geq (1-\eps^*)(1- \eps^2)\,p^{|J_x|}|V_{i_t\,j}| - \eps p^{|J_x|}|V_{i_t\,j}| \geq (1- 3\eps^*)\,p^{|J_x|}|V_{i_t\,j}|\\
     &\geq (1-\eps_{\BL})\,p^{|J_x|}|V_{i_t\,j}| \geq
            \bigl[(1-\eps_{\BL})p\,\bigr]^{|J_x|}|V'_{i_t\,j}|\,.
\end{align*}
Similarly, as $V'_{i_t\,j} \subseteq A_j^t \cup S$, we obtain 
\begin{align*}
     \bigl|N^*_{\Gamma}(J_x; V'_{i_t\,j})\bigr| & \leq |N^*_{\Gamma}(J_x; A^t_j)| + |N^*_{\Gamma}(J_x; S)| \leq (1+\eps^*)\,p^{|J_x|}|V_{i_t\,j}| + \eps p^{|J_x|}|V_{i_t\,j}|\\
     &\overset{\eqref{eq:Vprime_lb}}{\leq} (1+2\eps^*)p^{|J_x|}\cdot(1 - \eps^2)^{-1}|V'_{i_t\,j}| \leq (1+2\eps^*)(1 
     + 2 \eps)\, p^{|J_x|}|V'_{i_t\,j}|\\
     &\leq (1+\eps_{\BL})\,p^{|J_x|}|V'_{i_t\,j}| \leq \bigl[(1+\eps_{\BL})p\,\bigr]^{|J_x|}|V'_{i_t\,j}|\,,
\end{align*}
thereby establishing~\ref{itm:IR5}.

Finally, for~\ref{itm:IR6}, suppose~$x \in X_{i_t\,j}^*$ is given along with a vertex~$y \in V(H')$ such that~$xy \in E(H')$. Then, we have~$y \in N_{H'}[X_I]$. Moreover, since by construction, the set $N_{H'}[X_I]$ was partitioned into parts $N_{ij} \subseteq X'_{ij}$ to form an $R$-partition, it follows that $y \in X'_{i_t\,j'}$ for some $j' \in [K]$ and $j' \neq j$. Further, as there is no edge in $H$ between the sets $X_I^t$ and $X_I^{t'}$, it follows that either $y \in X_I^t$ or $y \notin X_I$. In either case, we have that~$J_y \subseteq \varphi_{t^*}(N_H(x_t))$.

As both $J_x$ and $J_y$ are subsets of $\varphi_{t^*}(N_H(x_t))$, property~\ref{itm:W4} is applicable. As~$xy$ is an edge in $H'$ and $\Delta(H) \leq \Delta$, we have $|J_x|, |J_y| < \Delta$. Moreover, if $\Delta = 2$, then $J_x \cap J_y = \emptyset$; indeed for otherwise, given any $z \in J_x \cap J_y$, the vertex~$\varphi_{t^*}^{-1}(z)$ would be adjacent to at least three vertices in $H$, namely $x$, $y$, and $x_t$. Thus, by~\ref{itm:W4}, it follows that $\bigl(N^*_{\Gamma}(J_x; A^t_j), N^*_{\Gamma}(J_y; A^t_{j'}) \bigr)$ is $(\eps^*, d, p)$-regular in~$G$. Note that~\ref{itm:W4} allows $J_y$ to be the empty set, in which case, $N^*_{\Gamma}(J_y; A^t_{j'} ) = A^t_{j'}$. From~\ref{itm:W3} we know that~$|N^*_{\Gamma}(J_x; A^t_j)| \geq (1-\eps^*)p^{|J_x|} |A^t_j|$\, and so by~$A^t_j \symmdiff V'_{i_t\,j} \subseteq S \cap V_{i_t\,j}$, we obtain
\begin{align*}
  \bigl|N^*_{\Gamma}(J_x; A^t_j) &\symmdiff N^*_{\Gamma}(J_x;V'_{i_t\,j})\bigr|
  =|N^*_{\Gamma}(J_x; A^t_j\symmdiff V'_{i_t\,j})| \leq
    |N^*_{\Gamma}(J_x; S)| \overset{\eqref{eq:Jx_S}}{\leq} 10\mu p^{|J_x|} n\\
    &\overset{ \text{\ref{itm:RG1} }}{\leq} 40\, \eps^2p^{|J_x|}|V_{i_t\,j}| 
    \overset{\eqref{eq:Atj_Vtj}}{\leq} 80\, \eps^2p^{|J_x|}|A^t_j| \overset{\text{\ref{itm:W3}}}{\leq} \frac{80\eps^2}{1
    - \eps^*}|N^*_{\Gamma}(J_x; A^t_j)| \leq \eps |N^*_{\Gamma}(J_x;
    A^t_j)| \,.
\end{align*}
The above equation holds analogously for $y$ when $J_y \neq \emptyset$. On the other hand, when $J_y = \emptyset$, then using~\ref{itm:RG1} and equation~\eqref{eq:Atj_Vtj}, we obtain
\begin{align*}
     \bigl|N^*_{\Gamma}(J_y; A^t_{j'}) &\symmdiff N^*_{\Gamma}(J_y;V'_{i_t\,j'})\bigr| = |A^t_{j'} \symmdiff V'_{i_t\,j'}| \leq |S \cap V_{i_t\, j'}|  \leq |S| \\
     & = \mu n = \eps^2 n/r \overset{\text{\ref{itm:RG1}}}{\leq} 4 \eps^2 |V_{i_t\, j'}| \overset{\eqref{eq:Atj_Vtj}}{\leq} 8 \eps^2|A^t_{j'}| \leq \eps \bigl|N^*_{\Gamma}(J_y; A^t_{j'})\bigr|  
\end{align*}
Now, since $\sqrt{\eps}, \eps^* \leq 0.1\, \eps_{\BL}$, it follows from Lemma~\ref{lem:sparse_reg_robust} that~$\bigl(N^*_{\Gamma}(J_x; V'_{i_t\,j}), N^*_{\Gamma}(J_y; V'_{i_t\,j'}) \bigr)$ is $(\eps_{\BL}, d, p)$-regular in $G'$, as required for~\ref{itm:IR6}. This concludes the verification of~\ref{itm:SBL3}. 

Now, since~\ref{itm:SBL1}--\ref{itm:SBL3} hold, the good event of~Lemma~\ref{lem:rg_blowup} returns an embedding~$\varphi'$ of~$H'$ into~$G'$ that agrees with the image restrictions~$\mathcal{I} = \{I_x\}$ generated by the pre-embedding
process.  Hence, the map~$\varphi \coloneqq \varphi_{t^*} \sqcup \varphi'$ gives an embedding of
the graph~$H$ into~$G$, as required.
\end{proof}

\section{Concluding Remarks}
\label{sec:concluding-remark}

In this paper we proved that the local resilience
of~$G(n,p)$ for the property of containing any spanning subgraph~$H$ having
maximum degree at most~$\Delta$ and with~$\Omega\,(p^{-2})$ vertices not contained in triangles is at
least~$1 - \exth$ for $p\geq C\,(\log n/n)^{1/\Delta}$. 

In~\cite{AllBoeKohNev:robust}, we prove that $\exth \leq 1 - 1/2\Delta$, making the local resilience at least $1/2\Delta$. On the other hand, by the construction witnessing the tightness of the  Bollob\'as--Eldridge--Catlin Conjecture, the local resilience could be as large
as~$1/(\Delta+1)$. Any improvement in the minimum degree condition of Theorem~\ref{thm:Sauer_Spencer} that additionally satisfies the extendability property in the definition of~$\exth$ (Definition~\ref{defn:extension_threshold}) will automatically improve the local resilience bound of Theorem~\ref{thm:resil_main_basic}.

We do
not believe our hypothesis on~$p$ in Theorem~\ref{thm:resil_main} to be optimal. As mentioned earlier, our hypothesis
matches the condition on~$p$ in~\cite{DKRR, Univ_Delta2} that ensures that~$G(n,p)$ is universal for spanning graphs having maximum degree at most~$\Delta$. This seems to be a natural benchmark for such embedding results. Indeed, for this range of~$p$, every $\Delta$-set of vertices in~$G(n,p)$ has a common neighbour with high probability, making this property a crucial component for several embedding algorithms in this regime.
We do not expect it to be easy to weaken our hypothesis on~$p$
substantially. Even in the well-studied setting of universality, this benchmark was only recently surpassed for~$\Delta \geq 3$ by Ferber and Nenadov~\cite{ferb_nena_univ}, who require the weaker bound~$p \geq (n^{-1} \log^3n)^{1/(\Delta - 1/2)}$. Any improvement to the hypothesis on~p in Theorem~\ref{thm:resil_main} would be interesting; in particular, whether it can be sharpened for $\Delta \geq 3$ to match the bound in~\cite{ferb_nena_univ}.

Theorem~\ref{thm:resil_main} restricts the class of bounded degree graphs~$H$ by additionally requiring the existence of vertices not contained in
triangles. As shown by Huang, Lee, and Sudakov~\cite{HLS_vtsintriangles}, this requirement in Theorem~\ref{thm:resil_main} is both necessary and optimal. However, this leads to the question whether one can further restrict the graph~$G$ naturally to obtain a local resilience theorem for the containment of all bounded degree graphs~$H$. A possible restriction on~$G$, as used in \cite{AllBoeEhrSchTar}, is to require that, in addition to the conditions of Theorem~\ref{thm:resil_main}, the graph~$G$ retains a positive proportion of the copies of $K_{\Delta+1}$ in~$\Gamma$ at each vertex~$v\in V(G)$. It would be interesting to see if this restriction generalises Theorem~\ref{thm:resil_main} to the class of all bounded degree spanning graphs, and if there are other ways to naturally restrict the graph $G$ to achieve the same.

Finally, we note that it would also be of interest
to transfer the Sauer--Spencer theorem to the \textit{pseudorandom}
setting of $(n,d,\lambda)$-graphs or `bijumbled graphs' having no
vertices of small degree.

\bibliographystyle{ams_edited}
\bibliography{sauerspencer-mrefed}

\appendix

\clearpage
\section{Remarks on the Statement of Lemma~\ref{lem:resil_G}}
\label{app:adaptLemG}

Lemma 26 from~\cite{sparse_bandwidth}, for the setting $r_0 = 1$ and with the omission of outcomes referring to backbone graphs $(B^K_r)$, can be stated as follows.

\begin{lemma}[Lemma 26~\cite{sparse_bandwidth}] 
  \label{lem:actual_resil_G}
  For each $\gamma > 0$ and integer $K \geq 2$, there exists $d_0 > 0$
  such that for every $\eps \in (0,1/2K)$, there
  exist $r_1\geq 1$ and $C_{\LG}>0$ such that the
  following holds a.a.s. for~$\Gamma \sim G(n,p)$ if
  $p \geq C_{\LG} (\log n/n)^{1/2}$.

  Let~$G=(V,E)$ be a spanning subgraph of $\Gamma$ with
  $\delta(G) \geq (1 - 1/K + \gamma)pn$. Then there exist
  an integer $m$ with $r = mK \leq r_1$; a subset $V_0 \subseteq V$
  with $|V_0| \leq C_{\LG}p^{-2}$; a $K$-equitable vertex partition~$\mathcal{V} = \{V_{ij}\}_{i\in[m],j\in[K]}$ of $V(G)\setminus V_0$;
  and an $r$-vertex graph $R$ on the vertex set $[m] \times [K]$ 
  with~$\delta(R) \geq (1 -1/K + \gamma/2)r$ and
  containing a $K$-clique factor $R'$ such that the following hold:
  \begin{enumerate}[label=\itmarab{RGO}]
  \item \label{itm:RGO1} $n/4r\leq
    |V_{ij}| \leq4n/r$ for every $i\in[m]$ and $j\in[K]$,
    
  \item \label{itm:RGO2} $(G,\mathcal{V})$ is
    $(\eps,d_0,p)$-lower-regular on $R$ and $(\eps,d_0,p)$-super-regular
    on $R'$,
    
  \item \label{itm:RGO3} the pairs $\bigl(N_\Gamma(v; V_{ij}),V_{i'j'}\bigr)$ and
    $\bigl(N_\Gamma(v; V_{ij}),N_\Gamma(v; V_{i'j'})\bigr)$ are
    $(\eps,d_0,p)$-lower-regular pairs in~$G$ for every
    $\{(ij),(i'j')\} \in E(R)$ and $v\in V\setminus V_0$, and
    
  \item \label{itm:RGO4} $|N_\Gamma(v;V_{ij})| = (1 \pm
    \eps)p|V_{ij}|$ for every $i \in [m]$, $j\in [K]$, and $v
    \in V \setminus V_0$.
    
  \end{enumerate}
\end{lemma}

This lemma can be morphed to Lemma~\ref{lem:resil_G} as follows. First, note that if \ref{itm:RGO1}--\ref{itm:RGO4} are true for $d_0 > 0$, then they also hold for any $d \leq d_0$ by the definitions of lower-regularity and super-lower-regularity. Thus, the outcome \ref{itm:RG1}--\ref{itm:RG4} can be required to hold for all $d \leq d_0$. 

Secondly, while the condition $|V_{ij}| \geq (1- 2\eps)n/r$, as mentioned in~\ref{itm:RG1}, is not part
  of~\ref{itm:RGO1}, it can be deduced from the
  proof of \mcite[Lemma 26]{sparse_bandwidth} as follows. Given $\eps,\, \gamma > 0$ and
  $K \geq 2$, as in the statement of the lemma, an auxiliary constant
  $\eps^*$ is chosen in the proof such that
  $\eps^* \leq 10^{-10}\eps^2\gamma K^{-2}$. The proof then begins
  with an equitable partition $\mathcal{U} = \{U_{ij}\}$ of
  $V(G)\setminus U_0$, where $U_0$ is an exceptional set of vertices
  of size at most $\eps^*n$. The required partition
  $\mathcal{V} = \{V_{ij}\}$ is then obtained from $\mathcal{U}$ by a
  redistribution of some vertices of $G$. As claimed in equation (4.3)
  of that proof, and by the choice of~$\eps^*$, we 
  have~$|U_{ij}\symmdiff V_{ij}| \leq 2000 K^2 \eps^*
  \gamma^{-1}|U_{ij}| \leq \eps|U_{ij}|$. Thus, we obtain \[|V_{ij}| \geq (1- \eps)|U_{ij}| \geq (1-\eps)(1-\eps^*)n/r \geq (1-
  2\eps)n/r,\] where the second inequality follows from the size of
  $U_0$ and equitability of the partition~$\mathcal{U}$.

  Finally, the minimum degree condition $\delta(G) \geq (1 - 1/K + \gamma)pn$ on $G$ can be generalised to $\delta(G) \geq (\xi + \gamma)pn$ for any $\xi \geq 1 - 1/K$ by replacing the input variable~$(K-1)/K + \gamma$ with~$\xi + \gamma$ in the application of \mcite[Lemma 13]{sparse_bandwidth} within the proof of \mcite[Lemma 26]{sparse_bandwidth}. This returns a reduced graph~$R$ with $\delta(R) \geq (\xi + \gamma/2)r \geq (1 - 1/K + \gamma/2)r$, as required, and the remainder of the proof of Lemma~$26$ proceeds unaffected.

\end{document}